\documentclass[a4paper,12pt,leqno]{article}
\usepackage{amsmath}
\usepackage{amsthm}
\usepackage{amssymb}
\usepackage{amscd}
\usepackage{graphicx}

\usepackage{a4wide}
\usepackage[all]{xy}
\newtheorem{thm}{Theorem}[section]

\newtheorem{prop}[thm]{Proposition}
\newtheorem{fact}[thm]{Fact}
\newtheorem{lemma}[thm]{Lemma}
\newtheorem{cor}[thm]{Corollary}
\newcommand{\Hol}{\mbox{{\rm Hol}}}
\newcommand{\Alg}{\mbox{{\rm Alg}}}

\newcommand{\SZ}{{\cal Z}^{\Delta}}

\newcommand{\Gr}{\Bbb G \mbox{r}}
\newcommand{\Z}{\Bbb Z}
\newcommand{\C}{\Bbb C}
\newcommand{\R}{\Bbb R}

\newcommand{\K}{\Bbb K}
\newcommand{\G}{\Bbb G}
\renewcommand{\P}{{\rm P}}

\newcommand{\RP}{\Bbb R\mbox{{\rm P}}}

\newcommand{\M}{\mbox{{\rm M}}}
\newcommand{\Map}{\mbox{{\rm Map}}}
\newcommand{\Rat}{\mbox{{\rm Rat}}}
\newcommand{\CP}{\Bbb C {\rm P}}

\newcommand{\dis}{\displaystyle}
\newcommand{\p}{\prime}

\newcommand{\E}{\tilde{E}}
\newcommand{\I}{\mbox{{\rm (i)}}}
\newcommand{\II}{\mbox{{\rm (ii)}}}
\newcommand{\III}{\mbox{{\rm (iii)}}}

\newcommand{\mo}{\mbox{{\rm (mod $2$)}}}

\title{
{\bf
Spaces of algebraic and continuous maps between  real algebraic varieties
}}
\author{{\it by }\bf 
Michal Adamaszek\footnote{
E-mail: aszek@mimuw.edu.pl}
\\
\normalsize{\it $($Department of Math., Univ. of Warsaw, Warsaw, Poland$)$}
\\
\normalsize{\it $($Warwick Math. Institute and DIMAP, Univ. of Warwick, Coventry, UK$)$}
\\
\bf 
Andrzej Kozlowski\footnote{
E-mail: andrzej@akikoz.net}
\\
\normalsize{\it $($Department of Math., Tokyo Denki University,
Chiba, Japan$)$}
\\
\bf
Kohhei Yamaguchi\footnote{
E-mail: kohhei@im.uec.ac.jp}
\\
\normalsize{\it $($Department of Math.,
Univ. Electro-Commun., Tokyo, Japan$)$}
}
\date{}

\begin{document}
\maketitle

\begin{abstract}
We consider the inclusion of the space of algebraic (regular) maps 
between real algebraic varieties
in the space 
of all continuous maps. For a certain class of real algebraic varieties, which include real projective spaces,  it is well known that the space of real algebraic maps is a dense subset of the space of all continuous maps. 
Our first result shows that, for this class of varieties,  the inclusion is also
a homotopy equivalence. After proving this, we restrict the class of varieties to real projective spaces.
In this case, the space of algebraic maps
has a  \lq minimum degree\rq\   
filtration by finite dimensional subspaces and it is natural to expect that the homotopy types of the terms of the filtration approximate closer 
and closer  the homotopy type of the space of continuous mappings 
as the degree increases. We prove this and compute the lower bounds of this approximation 
of these spaces. 
This result can be seen as a generalization of the results of Mostovoy, Vassiliev and others
on the topology of the space of real rational maps and the space of real polynomials without $n$-fold roots.  It can also  viewed as real analogue of Mostovoy's work on the the topology of the space of holomorphic maps
between complex projective spaces, 
which generalizes Segal's work 
on the space of complex rational maps.

\end{abstract}
\renewcommand{\thefootnote}{\fnsymbol{footnote}}


\section{Introduction.}
Let $X$ and $Y$ be two topological spaces 
with some additional structures, 
e.g. that of a complex or symplectic manifold or algebraic variety.  
The space $\mathcal{S} (X,Y)$ of continuous maps $X \to Y$ 
preserving the structure 
is a subspace of the space of all continuous maps $\Map(X,Y)$ and it is natural to ask if  the two spaces are in some topological sense (e.g. homotopy, homology) equivalent. Another reason to ask for such equivalence is that the space $\mathcal{S} (X,Y)$ might then provide a smaller homotopy or homology model for the space of all maps. Early examples of this type of phenomenon can be found in \cite{Grom}.
In many cases of interest the infinite dimensional space  $\mathcal{S} (X,Y)$  
has a filtration by finite dimensional subspaces, 
given by some kind of  \lq\lq map degree\rq\rq, 
and  the topology of these finite dimensional spaces 
approximates the topology of the entire space of continuous maps;
the approximation becoming more accurate as the degree increases. 
In recent years a great deal of attention has been 
devoted to studying problems of the above kind in the case 
where the  \lq\lq additional structure\rq\rq\  
in question is the structure of a complex manifold 
(so that the structure preserving maps are holomorphic maps), 
with the space $X$  a compact surface and the space $Y$ a certain complex manifold. 
The first explicit result of this kind seems to have been the following theorem of Segal:

\begin{thm}
[G. Segal, \cite{Se}]
\label{thm: S}
If $M_g$ is a compact closed Riemann surface of genus $g$, the inclusions 
$$
\begin{cases}
j_{d,\C}:\Hol_d(M_g,\CP^n)\to \Map_d(M_g,\CP^n)
\\
i_{d,\C}:\Hol_d^*(M_g,\CP^n)\to \Map_d^*(M_g,\CP^n)
\end{cases}
$$
are homology equivalences through dimension $(2n-1)(d-2g)-1$
if $g\geq 1$ and homotopy equivalences through dimension $(2n-1)d-1$ if $g=0$.
\par
Here $\Hol_d^*(M_g,\CP^n)$ $($resp. $\Hol_d(M_g,\CP^n))$
denote the spaces consisting of all
based $($resp. free$)$ holomorphic and $\Map_d^*(M_g,\CP^n)$ $($resp. $\Map_d(M_g,\CP^n))$  of all
based $($resp. free$)$
continuous maps $f:M_g\to \CP^n$ of degree $d$.
\end{thm}
Recall that a map $f:X\to Y$ is called {\it a homology $($resp. homotopy$)$
equivalence through dimension }$N$ if the induced homomorphism
$
f_*:H_k(X,\Z)\to H_k(Y,\Z)$
(resp.
$f_*:\pi_k(X)\to\pi_k(Y))$
is an isomorphism for all $k\leq N$.

Segal conjectured  in \cite{Se} that 
this result should generalize to a much larger class of target spaces, 
such as complex  Grassmannians and  flag manifolds, 
and even possibly to higher dimensional source spaces. Almost all of the work inspired by Segal's results  has been  concerned  with extending the Segal results to a larger class of target spaces, while the source space has been kept (complex) one dimensional
(e.g. \cite{BHMM}, \cite{Gr}, \cite{Gu}, \cite{Gu1}, \cite{GKY1}, \cite{Ki}).
Several authors (e.g.   \cite{BHM},  \cite{CJS}) have attempted to find the most general target spaces for which the stability theorem holds. There have been, however, very few  attempts  to investigate (as suggested by Segal) the phenomenon of topological stability for source spaces of complex dimension greater than 1.  The first steps in this direction were taken by Havlicek \cite{Hav} who considered the space of holomorphic maps from $\CP^1\times \CP^1$ to complex Grassmanians and Kozlowski and Yamaguchi \cite{KY1}  who studied the case of linear maps  $\CP^m \to \CP^n$, where $m\leq n$. 
A major step was taken when Mostovoy \cite{Mo2} published a proof of the following analogue of Segal's theorem for the space of holomorphic maps  
from $\CP^m$ to $\CP^n$. 

\begin{thm}
[J. Mostovoy, \cite{Mo2}]
\label{thm: M}
If $2\leq m\leq n$ and $d\geq 1$ are integers,
the inclusion
$j_{d,\C}:\Hol_d(\CP^m,\CP^n)\to \Map_d(\CP^m,\CP^n)$
is a homotopy equivalence through dimension $(2n-2m+1)\Big(\big\lfloor\frac{d+1}{2}\big\rfloor+1\Big)-1$
if $m<n$, and a homology equivalence through the same dimension if $m=n$.
\end{thm}

\par\vspace{2mm}\par

A notable  feature of the argument in \cite{Mo2}  is that it  uses  real rather than complex methods. This suggests that a similar method could be applicable to proving a real analogue of Mostovoy's theorem. As explained in \cite{Se}, the original motivation for Segal's work on rational functions came from the real case. The first \lq\lq real analogue\rq\rq\  of Segal's theorem was proved  by Segal himself in the same work.  A different  real analogue of Segal's Theorem  was proved in \cite{Mo1}, \cite{Va} and  \cite{GKY2}.
It is this result that we generalize here. 

\par\vspace{2mm}\par

Unfortunately there were some gaps and errors in Mostovoy's published proof.  When the first version of this article was written they  still had not been corrected, but we have since obtained  from the author a new preprint \cite{Mo3} which we believe to be correct. Many ideas from Mostovoy's work play key roles in our work. The errors and gaps in the published version of \cite{Mo2} do not have a significant effect on the real case, which we are considering here. However, after seeing the revised version of \cite{Mo3} and following a number of suggestions of the referee we have made a number of changes, which enabled us to improve some of our results. \footnote{ A new version of the paper, currently only available from the author, appears to correct all the mistakes, with the main results remaining essentially unchanged. There are two major changes in the proofs. One is that the space $\Rat_f(p,q)$ of $(p,q)$ maps from $\CP^m$ to $\CP^n$  that restrict to a fixed map $f$ on a fixed hyperplane, used in section 2 of the published article is replaced by the space $\overline {\Rat}_f(p,q)$ of pairs of $n+1$-tuples of polynomials in $m$ variables that produce these maps. In the published version of the article it is assumed that these two spaces are homotopy equivalent, which is clearly not the case. However, they are homotopy equivalent after stabilisation, both being equivalent to $\Omega^{2m}\CP^n$ - the space of continuous maps that restrict to $f$ on a fixed hyperplane.
 The second important change is the introduction of a new filtration on the simplicial resolution $X^\Delta \subset \Bbb R^N\times Y $ of a map $h: X \to Y$ and an embedding $i: X \to \Bbb R^N = \Bbb C^{N/2}$. This filtration is defined by means of complex skeleta  (where the complex $k$-skeleton of a simplex in a complex affine space is the union of all its faces that are contained in complex affine subspaces of dimension at most $k$) and replaces the analogous \lq\lq real\rq\rq\ filtration in the arguments of section 4.}

\paragraph{Overview of our work.}
 In this paper we derive a real analogue of the above theorems. In general,  we study the inclusions of the form
$$
\Alg (\RP^m,\RP^n)\stackrel{\subset}{\longrightarrow} \Map (\RP^m,\RP^n)
$$
of the subspaces of algebraic maps in the space of continuous maps between \emph{real} projective spaces. Our results consist of three fairly independent parts:
\begin{itemize}
\item \textbf{Stable results.} We provide a fairly general stable result for the inclusion of algebraic maps in the space of all continuous maps between real algebraic varieties (we do not restrict to projective spaces in this part). This is the topic of Section \ref{section:stable} and the main result is the Stable Theorem \ref{thm: I}.
\item \textbf{Finite-dimensional approximation.} Next, we focus attention only on maps between real projective spaces, and we construct a family of finite dimensional spaces of polynomials which approximate the space of continuous maps more and more accurately. These spaces consist simply of collections of polynomials that represent algebraic maps of various (increasing) degrees. The relevant definitions and precise statement of results are contained in Sections \ref{section:notation1} and \ref{section:notation2}, while the results are proved in Section \ref{section 4}.
\item \textbf{Subspaces of algebraic maps.} Finally we consider the possibility of approximating the space of continuous maps by its subspaces of algebraic maps of a fixed degree, analogously to the results of Segal and Mostovoy. Such approximation in homology is deduced from the results of the previous part, along with a homology equivalence between the spaces of polynomials and spaces of maps they represent (see Theorem \ref{thm: A4}). These results are introduced in Sections \ref{section:notation1} and \ref{section:notation3} and the proofs are in Section \ref{section 5}. The Appendix describes an upgrade from homology equivalence to homotopy equivalence when the source space is one-dimensional.
\end{itemize}

\section{Stable real results.}
\label{section:stable}


We state the stable result of this section with much generality. Let $\Bbb K$ denote the fields $\R$ or $\C$ of real or complex
numbers.

Let $\G_{n,k}(\Bbb K)$ be the Grassmanian manifold
of $k$ dimensional $\Bbb K$-subspaces in $\Bbb K^n$.
For an affine real algebraic variety $X$, let
$\Alg (X,\G_{n,k}(\Bbb K))$ denote the space consisting of all algebraic
(i.e. regular) maps $f:X\to \G_{n,k}(\Bbb K)$ and $\Map(X,\G_{n,k}(\Bbb K) )$ the spaces of continuous maps.

\begin{thm}[{\bf Stable Theorem}]\label{thm: I}
Let $X$ be a compact affine real algebraic  variety, with the property that every topological $\Bbb K$-vector bundle of rank $k$ over $X$ is topologically isomorphic to an algebraic $\Bbb K$-vector bundle. 
Then the inclusion 
$i: \Alg (X,\Bbb \G_{n,k}(\Bbb K))\to \Map(X,\G_{n,k}(\Bbb K) )$ is a weak homotopy equivalence.
\end{thm}

Note that the assumption of Theorem \ref{thm: I} is satisfied for 
the  $\K$-projective space $X=\K\P^m$.
%
%
We start by reminding the definition of 
an \lq\lq algebraic map\rq\rq\, usually referred to in algebraic geometry as a \lq\lq  regular\rq\rq\  map.

\par\vspace{2mm}
\par\noindent{\bf Definition. }
Let $V\subset \K^n$ be an algebraic subset and  $U$ be a (Zariski) open  
subset of $V.$ We say that a function $f:U \to \K$ is 
{\it an regular function on} $U$ 
if it can be written as the quotient of two polynomials $f=g/h$, with $h^{-1}(0)\cap U=\emptyset$. 
For a subset $W\subset \K^p$, a map $\varphi :U\to W$ is called
{\it an algebraic map} if its coordinate functions are regular functions.

\par\vspace{2mm}\par
Clearly regular functions on an algebraic set  form a sheaf. 
An  regular map between two algebraic sets is 
one that induces a map of sheaves or regular functions (is a morphism of ringed spaces).  
These definitions extend  in a natural way to abstract algebraic varieties, 
and in particular to projective varieties.  
An algebraic map between real or complex algebraic varieties 
is a continuous map (with respect to the complex or real topology) and a $C^\infty$ map in the case of smooth varieties.
\par

To state our stable results  in the most general form we need the concept 
of an algebraic $\Bbb K$-vector bundle 
over a real algebraic variety.
\par
Recall (\cite{BCR}) that {\it a pre-algebraic vector bundle} over a real algebraic variety $X$ is a triple $\xi =(E,p,X)$, such that $E$ is a real algebraic variety, $p:E \to X$ is an algebraic map, the fiber over each point is a $\Bbb K$-vector space and there is a covering of the base $X$ by Zariski open sets over which the vector bundle $E$ is biregularly isomorphic to the trivial bundle.
{\it An algebraic vector bundle} over $X$ 
is a pre-algebraic vector bundle which is algebraically isomorphic to a pre-algebraic vector sub-bundle of a trivial bundle. 
Most of the usual vector bundle constructions when performed on algebraic vector  bundles give rise to algebraic vector bundles;  in particular the Whitney sum, and the tensor product of two algebraic bundles is algebraic, so is the pull-back of an algebraic bundle by a regular map between real algebraic varieties. 
Important examples of algebraic vector bundles are the universal vector bundles over the Grassmanian manifold 
$\Bbb G_{n,k}(\Bbb K)$ of all $k$ dimensional $\K$-subspace in
$\K^n$
and their complementary bundles. 

\begin{lemma}[\cite {BCR}]
\label{lemma: A}
 A topological vector bundle over a real algebraic variety $X$, which is stably isomorphic to an algebraic vector bundle, is topologically isomorphic to an algebraic vector bundle.
\end{lemma}
The above lemma and the Stone-Weierstrass theorem are the main ingredients in proving the following result:
\begin{thm}[\cite{BCR}]\label{thm: B}
Let $X$ be a compact affine real algebraic variety, with the property that every topological $\Bbb K$-vector bundle of rank $k$ over $X$ is topologically isomorphic to an algebraic $\Bbb K$-vector bundle. 
Then the space of algebraic mappings $\Alg (X,\Bbb \G_{n,k}(\Bbb K))$ is dense in $\Map(X,\G_{n,k}(\Bbb K))$.
\end{thm}
Note that we are considering $\Bbb G_{n,k}(\Bbb C)$ as a real affine algebraic variety (see  \cite{BCR}).
Using essentially the same method as in \cite{BCR} and an idea from \cite{Mo2} we will prove our  Theorem \ref{thm: I}, which asserts that, under the assumptions of Theorem \ref{thm: B}, the space of algebraic maps is not only dense in the space of continuous maps but is also homotopy equivalent to it.

Note that both spaces $\Bbb \RP^m$ and $\Bbb  \CP^m$ satisfy 
the assumption  of  Theorem \ref{thm: B}.
The real case, that is the most important for us here,  follows from  Lemma \ref{lemma: A}  
and from the fact the $\tilde {KO}(\RP^m)$ is generated by the class of a unique non-trivial line bundle (\cite{Huse}, Theorem 12.7).  
The other cases follow form known results about $\tilde {KO}(\CP^m)$, $\tilde {K}(\CP^m)$ and $\tilde {K}(\RP^m)$.  

The assumption of Theorem \ref{thm: I} is not satisfied in general, but it is known that for every compact smooth manifold $M$, there exists a non-singular real algebraic variety $X$ diffeomorphic to $M$ such that every topological vector bundle over $X$ is isomorphic to a real algebraic one (\cite{BCR}). 

To prove Theorem  \ref{thm: I} we use the following Proposition: 

\begin{prop}\label{prop:D}
Let $X$ be as in Theorem \ref{thm: B} and let $Y$ be a finite CW-complex. 
Let $F: Y\times X \to \Bbb G_{n,k}(\Bbb K)$ be a continuous map. 
Then $F$ can be approximated uniformly (with respect to some metric on $ \Bbb G_{n,k}(\Bbb K)$) by maps 
$G: Y\times X \to \Bbb G_{n,k}(\Bbb K)$, such that the restriction 
$G_y:\{y\}\times X \to \Bbb G_{n,k}(\Bbb K)$ is a algebraic map for each $y\in Y$. 
\end{prop}
\begin{proof}
We imitate the method of proof of Lemma 4 of \cite{Mo2}.
Elements of  $\Map(Z,\Bbb G_{n,k}(\Bbb K))$  are in one to one correspondence with with  objects of the form:  $k$-dimensional $\K$-bundle $E$ on $Z$ together with  $n$ sections, which, at each point $z$ of  $Z$ span the fiber $E_z$ of the bundle. Let $f: Y\times X \to \Bbb G_{n,k}(\Bbb K)$ be a continuous map and denote by $E_f$ and $s_1,\dots,s_n$ the the corresponding bundle on $Y\times X$  and sections spanning the fiber at each point.  
Since $Y$ is a finite complex,
we can choose an open contractible cover 
$\{U_{k}\}_{k=1}^l$ of $Y$
and and associated partition of unity 
$\{\rho_{k}:Y\to \R\}_{k=1}^l$.  
Since $U_k$ is contractible, because of the assumption on $X$, 
for each $k$ there is an algebraic bundle $P_k$ on $X$ 
such that the restriction of $E_f$ to $U_k\times X$ is isomorphic 
to a bundle induced from $P_k$ by the projection map. 
We can thus identify each such restriction with $P_k$. 
Let $\Sigma$ denote the family of sections of $E_f$ over $Y\times X$ 
generated by all sections of the form $\rho_{k} s$ in $U_k$, 
where $s$ is an algebraic section of $P_k$, 
extended to the whole of $Y\times X$ by $0$. 
It is easy to see that the family of sections 
$\Sigma$ satisfies the conditions of the Stone-Weierstrass theorem 
for  vector bundles \cite{BCR} so each section $s_i$ 
can be uniformly approximated by algebraic sections. 
By choosing sections close enough to the $s_i$ 
we can choose approximations 
consisting of $n$ algebraic sections spanning the fiber. 
These sections define the required algebraic map 
$Y\times X \to \Bbb G_{n,k}(\Bbb K)$. 
\end{proof}

\par\noindent{\bf Remark. }
If $X$ is $\RP^n$ or $\CP^n$ then essentially the  argument reduces to the one given in Lemma 2.2 of \cite{Mo2}. 
\par\vspace{2mm}\par

\begin{proof}[Proof of  Theorem  \ref{thm: I}]
Let $N$ be any positive integer. 
It suffices to show that the induced homomorphism
$i_*:\pi_N(\Alg(X,\Gr_{n,k}(\Bbb K)))\to \pi_N(\Map(X,\Gr_{n,k}(\Bbb K) ))$ is an isomorphism. 
 First, we will show that $i_*$ is injective. 
 Let  $\alpha = [F]\in \pi_N(\Map(X,\Gr_{n,k}(\Bbb K)))$ be any element and 
 let $\tilde F : S^N\times X \to \Gr_{n,k}(\Bbb K) $ be the adjoint of the representative map $F$. 
By Proposition \ref{prop:D}  
there is a uniform approximation to $\tilde F$,
  $\tilde G : S^N\times X\to \Gr_{n,k}(\Bbb K) $, 
  such that $\tilde G\vert X\times  \{s\}$ is a algebraic map for each $s\in S^N$.  
  Let $G: S^N \to \Alg(X,\Gr_{n,k}(\Bbb K)) $ denote the adjoint of $\tilde G$. Clearly, $i_*([G])=\alpha$. Hence $i_*$ is  surjective. 

Next, we show that $i_*$ is injective. 
For this purpose, suppose that $i_*(\beta_0) = i_*(\beta_1)$ for 
$\beta_l \in \pi_N(\Alg(X,\Gr_{n,k}(\Bbb K)))$ $(l=0,1)$. 
Let $\beta_l = [F_l]$ and let ${\tilde F}_l$  denote the adjoint of $F_l$ 
for $l \in \{0,1\}$. There is a homotopy 
$\Phi: X \times S^N \times [0,1] \to X$ such that 
$\Phi| X\times S^N \times \{l\} = {\tilde F}_k$ for $l \in \{0,1\}$. 

Then by Proposition \ref{prop:D} there is a uniform approximation $\Psi: X\times S^N\times [0,1] \to X$ to $\Phi$ such that $\Psi|X\times \{(s,t)\}$ is a algebraic map for any $(s,t)\in S^N\times [0,1]$. Hence, $\beta_0=\beta_1$ and $i_*$ is injective. 
\end{proof}


\section{Notation and introduction of unstable results.}
For our \lq\lq unstable\rq\rq\  results we limit ourselves to the case where on both sides  we have real projective spaces.

\subsection{Notation.}
\label{section:notation1}
Let $m$ and $n$ be positive integers such that $m<n$, and  
let $z_0,\cdots ,z_m$ denote variables.
We choose the point ${\bf e}_k=[1:0:\cdots :0]\in \RP^k$ as the base point of $\RP^k$. We also let $\RP^{k-1}\subset \RP^{k}$ denote the fixed subspace defined by the equation $z_k=0$.

We now need to define three types of spaces: spaces of continuous maps, spaces of algebraic maps and spaces of collections of polynomials representing the algebraic maps. Each of them comes in three flavors: free, basepoint-preserving and restricted. That means we are going to have $3\times 3=9$ types of spaces altogether.

\paragraph{Continuous maps. }
We denote by $\Map(\RP^m,\RP^n)$ and $\Map^*(\RP^m,\RP^n)$, respectively, the space of all continuous maps $f:\RP^m\to \RP^n$ and its subspace of basepoint-preserving maps which satisfy $f({\bf e}_m)={\bf e}_n$.

For $\epsilon \in \Z/2=\{0,1\}=
\pi_0(\Map(\RP^m,\RP^n))$, 
let $\Map_{\epsilon}(\RP^m,\RP^n)$ and 
$\Map_{\epsilon}^*(\RP^m,\RP^n)$
be the corresponding path components of these spaces.

Given a map $g\in\Map_{\epsilon}^*(\RP^{m-1},\RP^n)$, we can consider the space of all maps whose restriction to the fixed subspace coincides with $g$:
$$
F(m,n;g) =
\{f\in \Map_{\epsilon}^*(\RP^m,\RP^n):f\vert{\RP^{m-1}}=g\}.
$$
This is what we call the space of \emph{restricted maps}.
Now consider the restriction map
$r:\Map_{\epsilon}^*(\RP^m,\RP^n)\to
\Map_{\epsilon}^*(\RP^{m-1},\RP^n)$
given by the restriction $r(f)=f\vert \RP^{m-1}$.
It can be easily seen that 
there is a restriction fibration sequence 
$$
F(m,n;g)\hookrightarrow\Map_{\epsilon}^*(\RP^m,\RP^n)
\stackrel{r}{\longrightarrow}
\Map_{\epsilon}^*(\RP^{m-1},\RP^n).
$$
Observe that there is a homotopy equivalence
$F(m,n;g)\simeq  \Omega^mS^n$.

\paragraph{Algebraic maps. }
A map $f:\RP^m\to \RP^n$ is called {\it an algebraic map of degree} $d$
if it has a representation $f=[f_0:\cdots :f_n]$ where
$f_0,\cdots ,f_n\in \R [z_0,z_1,\cdots ,z_m]$
are homogenous polynomials of the same degree $d$
with no common {\it real} root other than
${\bf 0}_{m+1}=(0,\cdots ,0)\in\R^{m+1}$
(but possibly with common {\it non-real} roots).
The space of all such maps will be denoted $\Alg_d(\RP^m,\RP^n)$, while $\Alg_d^*(\RP^m,\RP^n)$ will be its subspace consisting of basepoint-preserving maps.

Note that we always have $\Alg_d(\RP^m,\RP^n) \subset \Alg_{d+2}(\RP^m,\RP^n)$
and $\Alg_d^*(\RP^m,\RP^n) \subset \Alg_{d+2}^*(\RP^m,\RP^n)$, 
because
$[f_0:f_1:\cdots :f_n]= 
[g_{m} f_0:g_{m} f_1:\cdots :g_{m} f_n]$, 
where $g_{m} = \sum_{k=0}^mz_k^2$. In fact, we could take $g_{m}$ to be any polynomial of degree $2$ with no nontrivial real root.
Clearly, we have an inclusion
$$j_{d,\R}:\Alg_d(\RP^m,\RP^n)\to\Map_{[d]_2}(\RP^m,\RP^n)$$
where $[d]_2\in \{0,1\}$ denotes the integer $d$ mod $2$. This map restricts to an inclusion
$$i_{d,\R}:\Alg_d^*(\RP^m,\RP^n)\to\Map_{[d]_2}^*(\RP^m,\RP^n).$$

In these terms the Stable Theorem \ref{thm: I} can be expressed as follows.
\begin{cor}
If $1\leq m< n$ and $\epsilon=0$ or $1$, the inclusion map
$$\bigcup_{k=1}^{\infty}\Alg_{\epsilon+2k}(\RP^m,\RP^n) \to \Map_{\epsilon}(\RP^m,\RP^n)$$
is a homotopy equivalence. The same holds for spaces of based maps.
\end{cor}
As before we shall consider restricted maps. 
For a fixed basepoint preserving algebraic map 
$g\in \Alg_d^*(\RP^{m-1},\RP^n)$, we let
$\Alg_d(m,n;g)$ be the space defined by
$$\Alg_d(m,n;g)=\{f\in \Alg_d^*(\RP^m,\RP^n):f\vert \RP^{m-1}=g\}.$$
Again, we see that the inclusion $i_{d,\R}$ restricts further to give an inclusion
$$i_{d,\R}^{\p}:\Alg_d(m,n;g)\to F(m,n;g).$$

\paragraph{Spaces of polynomials. }
Now we move on to the spaces of collections of polynomials which represent the algebraic maps. These will all be subsets of real affine or real projective spaces.
\par
Let ${\cal H}_{d,m}\subset \R [z_0,\cdots ,z_m]$
denote the subspace consisting of all
homogenous polynomials of degree $d$.
Let $ A_{d}(m,n)(\Bbb R)\subset {\cal H}_{d,m}^{n+1}$ denote the space of all $(n+1)$-tuples 
$(f_0,\cdots ,f_n)\in \R[z_0,\cdots ,z_m]^{n+1}$
of homogeneous polynomials of degree $d$ without non-trivial common real roots
(but possibly with non-trivial common {\it complex} roots). These are precisely the collections that represent algebraic maps. Since the algebraic map is invariant under multiplication of all polynomials by a non-zero scalar, it is convenient to define the projectivisation 
$$\tilde{A}_d(m,n)= A_{d}(m,n)(\Bbb R) / \R^*$$
which is a subset of the real projective space
$({\cal H}_{d,m}^{n+1}\setminus\{{\bf 0}\})/\R^*$.

We have a projection
$$
\Gamma_d:\tilde{A}_d(m,n)\to \Alg_d(\RP^m,\RP^n)
$$
which, composed with the inclusion $j_{d,\R}$ yields a map
$$
j_d:\tilde{A}_d(m,n)\to \Map_{[d]_2}(\RP^m,\RP^n)
$$
that takes a collection of polynomials to the map it represents.

To describe basepoint-preserving maps we use the space $A_{d}(m,n)\subset A_d(m,n)(\R)$, which consists of $(n+1)$-tuples $(f_0,\cdots ,f_n)\in A_d(m,n)(\R)$ such that the coefficient at $z_0^d$ is $1$ in $f_0$ and $0$ in the remaining $f_k$. Note that every algebraic map which preserves the basepoint has a representation by such a collection of polynomials.

We also get a natural projection map
$$\Psi_d:A_d(m,n)\to \Alg_d^*(\RP^m,\RP^n)$$
which induces a composite map
$$i_d=i_{d,\R}\circ \Psi_d:A_d(m,n)\to \Map^*_{[d]_2}(\RP^m,\RP^n).$$

Now we define polynomial representations of restricted maps.
Let $g\in \Alg_d^*(\RP^{m-1},\RP^n)$ be a fixed basepoint-preserving algebraic map 
with some fixed polynomial representation $g=[g_0:\cdots :g_n]$ such that 
$(g_0,\cdots ,g_n)\in A_d(m-1,n)$. Set $B_k=\{g_k+z_mh:h\in {\cal H}_{d-1,m}\}$ $(k=0,1,\ldots ,n)$ and let
$$A_d^*=B_0\times B_1\times \cdots \times B_n\subset  {\cal H}_{d,m}^{n+1}.$$
This space contains collections of polynomials which restrict to $(g_0,\cdots ,g_n)$ when $z_m=0$. Define  $A_d(m,n;g)\subset A_d^*$ to be the subspace of all collections with 
no non-trivial common {\it real} zero. Clearly, such a collection determines a map in $\Alg_d(m,n;g)$ and we again obtain a projection
$$\Psi_d^{\p}:A_d(m,n;g)\to \Alg_d(m,n;g)$$
and composite map
$$
i_d^{\p}:A_d(m,n;g)\to F(m,n;g).
$$

\paragraph{Remark.} The spaces of free and based mappings are common objects of interest in topology and have both been studied by Segal and subsequent authors. The third type --- restricted maps --- was introduced in \cite{Mo2}. This new space of mappings has some independent interest, since, it is homotopy equivalent to the well studied space of $m$-fold loops on the $n$-sphere (or $2m$ fold loops on the $2n+1$-sphere in the complex case). However, the principal reason for introducing these spaces in \cite{Mo2} is that they play a key role in an inductive proof of theorems concerning free maps. This will be similar in our approach. 

The definitions of this section can be summarized in the following useful diagram, for which we assume $g$ is a basepoint-preserving algebraic map of degree $d$.

\[
\xymatrix{%
\Map_{[d]_2}(\RP^m,\RP^n)     &     \Map^*_{[d]_2}(\RP^m,\RP^n) \ar@{_{(}->}[l]  &   F(m,n;g) \ar@{_{(}->}[l] \\
\Alg_d(\RP^m,\RP^n) \ar@{^{(}->}[u]^{j_{d,\R}}     &     \Alg_d^*(\RP^m,\RP^n) \ar@{^{(}->}[u]^{i_{d,\R}} \ar@{_{(}->}[l]  &   \Alg_d(m,n;g)  \ar@{^{(}->}[u]^{i_{d,\R}^{\p}} \ar@{_{(}->}[l]\\
\tilde{A}_d(m,n) \ar@{->>}[u]^{\Gamma_d} \ar@/^4pc/[uu]^>>>>>>{j_d}   &     A_d(m,n) \ar@{->>}[u]^{\Psi_d} \ar@/^4pc/[uu]^>>>>>>{i_d} &   A_d(m,n;g) \ar@{->>}[u]^{\Psi_d^{\p}} \ar@/^4pc/[uu]^>>>>>>{i_d^{\p}}
}
\]

We are also going to make use of the following:
\begin{fact}
If $m+2\leq n$ then $A_d(m,n)$, $A_d(m,n;g)$ and $\Map_\epsilon^*(\RP^m,\RP^n)$ are simply-connected.
\end{fact}
\begin{proof} 
Because
both $A_d(m,n)$ and $A_d(m,n;g)$ are of codimension $n-m+1$ in the affine space in which they are contained, they are simply connected.
For $\Map_\epsilon^*(\RP^m,\RP^n)$, 
the simply connectivity follows by analyzing the restriction fibration sequence
$\Omega^m S^n \to
\Map_\epsilon^*(\RP^m,\RP^n)\stackrel{r}{\rightarrow}
 \Map_\epsilon^*(\RP^{m-1},\RP^n)$.
\end{proof}

\subsection{Finite-dimensional approximations.}
\label{section:notation2}
In this subsection we state the theorems which can be viewed as a construction of finite-dimensional approximations to the spaces of continuous maps between real projective spaces. More precisely, we show that the various spaces of polynomials are homotopy (or homology) equivalent to their corresponding spaces of maps and that the equivalence range increases with the degree. Some statements of this sort have been known, for example

\begin{thm}[\cite{KY1}, \cite{Y0}]\label{thm: A2}
If $n\geq 2$ and $d\geq 1$,
the natural projection map
$i_d:A_d(1,n)\to  \Map^*_{[d]_2}(\RP^1,\RP^n)\simeq \Omega S^{n}$
is a homotopy equivalence
through dimension $(d+1)(n-1)-2$.
\end{thm}

Now we consider the problem of generalizing Theorem \ref{thm: A2} to $m\geq 2$.
For a connected space $X$, let $C_r(X)$ denote the space of
$r$ distinct unordered points in $X$, and let
$D(d;m,n)$ denote the positive integer given by
\begin{equation}\label{MD}
D(d;m,n)=(n-m)\big(\lfloor \frac{d+1}{2}\rfloor  +1\big) -1.
\end{equation}

\begin{thm}[\textbf{Finite dimensional approximation, restricted case}]\label{thm: II}
Let $2\leq m<n$ be integers and
$g\in\Alg_d^*(\RP^{m-1},\RP^n)$ be a fixed basepoint-preserving algebraic map.
\begin{enumerate}
\item[$\I$]
The 
map
$i_d^{\p}:A_d(m,n;g)\to F(m,n;g)\simeq \Omega^mS^n$
is a homotopy equivalence through dimension $D(d;m,n)$ if $m+2\leq n$
and a homology equivalence through dimension $D(d;m,n)$ if $m+1=n$. 
\item[$\II$]
For any $k\geq 1$, 
$H_k(A_d(m,n;g),\Z)$ contains the subgroup
$$
G^d_{m,n}=\bigoplus_{r=1}^{\lfloor \frac{d+1}{2}\rfloor}
H_{k-(n-m)r}(C_r(\R^m),(\pm \Z)^{\otimes (n-m)} )
$$
as a direct summand.
\end{enumerate}
\end{thm}

Using the method of \cite[Ch. 6]{Mo2} we can use the above to obtain a generalization of Theorem \ref{thm: A2}.
\begin{thm}[\textbf{Finite dimensional approximation, free and basepointed case}]\label{thm: III}
If $2\leq m <n$ and 
$d\geq 1$ are integers, 
the natural maps
$j_d : \tilde{A}_d(m,n)  \to \Map_{[d]_2}(\RP^m,\RP^n)$
and
$i_d  : A_d(m,n) \to \Map_{[d]_2}^*(\RP^m,\RP^n)$
are homotopy equivalences through dimension $D(d;m,n)$ if
$m+2\leq n$ and homology equivalences through dimension
$D(d;m,n)$ if $m+1=n$.
\end{thm}

Theorems \ref{thm: II} and \ref{thm: III} are proved in Section \ref{section 4}.
\par\vspace{1mm}\par
\noindent{\bf Remark.} We note here that, unlike in \cite{Mo2}, we actually do not use the \lq\lq stable\rq\rq\ results of Section \ref{section:stable} in our proof of our \lq\lq unstable\rq\rq ones. The price that we pay for this is that we have to use the main theorem of Vassiliev \cite{Va} describing the spectral sequence  for computing the cohomology of the space of  mappings from a $m$-dimensional CW-complex to a $n$-connected one (where $n\ge m$). Our \lq\lq unstable\rq\rq  results are obtained by directly comparing 
spectral sequences of Vassiliev type.

\subsection{Subspaces of algebraic maps.}
\label{section:notation3}
The purpose of this part is to obtain an extension of the following theorem:
\begin{thm}[J. Mostovoy, \cite{Mo1}]\label{thm: A1}
If $n\geq 2$ and $d\geq 1$ are integers,
the inclusion map $i_{d,\R}:\Alg_d^*(\RP^1,\RP^n)\to \Map_{[d]_2}^*(\RP^1,\RP^n)\simeq \Omega S^n$ is a homotopy equivalence through dimension 
$d(n-1)-2$. 
\end{thm}


Observe that the fibers of the natural surjection
$\Psi_d:A_d(m,n)\to \Alg_d^*(\RP^m,\RP^n)$
are contractible (indeed, each fiber is homeomorphic with the convex space of everywhere positive polynomials of degree $d$ in $(n+1)$ variables with leading coefficient $1$). However, this is not sufficient to prove that the map is a homotopy equivalence, and so the best comparison between them we can offer is the following:

\begin{thm}[{}]\label{thm: A4}
The natural projection maps $\Psi_d : A_{d}(m,n)\to \Alg_d^*(\RP^m,\RP^n)$ and $\Gamma_d:\tilde{A}_d(m,n)\to \Alg_d(\RP^m,\RP^n)$
are homology equivalences.
\end{thm}
Theorem \ref{thm: A4} is proved in Section \ref{section 5}.
\par
From this and Theorem \ref{thm: III} we immediately obtain the result on subspace inclusions:

\begin{thm}[\textbf{Homology approximation by subspaces}]\label{thm: IIIa}
If $2\leq m <n$ and 
$d\geq 1$ are integers, 
the inclusion maps
$j_{d,\R}:\Alg_d(\RP^m,\RP^n)\to \Map_{[d]_2}(\RP^m,\RP^n)$
and
$i_{d,\R}:\Alg_d^*(\RP^m,\RP^n)\to \Map_{[d]_2}^*(\RP^m,\RP^n)$
are homology equivalences through dimension $D(d;m,n)$.
\end{thm}
\noindent{\bf Remark.} 
To replace homology equivalences in the above by homotopy equivalences we would need to know that $\Alg_d^*(\RP^m,\RP^n)$ are simply connected for $m+1<n$. In fact the case $m=1$ we can replace \lq\lq homology\rq\rq\  with \lq\lq homotopy\rq\rq\  by a direct argument, which we sketch in the Appendix.


\section{Finite-dimensional approximations.}\label{section 4}
In this section, we construct a spectral sequence converging to
 $H_*(A_d(m,n;g),\Z)$, and use it to prove Theorem \ref{thm: II}. Then we deduce Theorem \ref{thm: III}.

\subsection{Simplicial resolutions.}
\par\noindent{\bf Definition. }
(i) For a finite set ${\bf x} =\{x_0,x_1,\cdots ,x_k\}\subset \R^N$,
let $\sigma ({\bf x})$ denote the convex hull spanned by ${\bf x}.$
Note that   $\sigma ({\bf x})$
is a $k$-dimensional simplex if and only if vectors
$\{x_j-x_0\}_{j=1}^k$ are linearly independent.
In particular, $\sigma ({\bf x})$ is a
$k$-dimensional simplex
 if $x_0,\cdots ,x_k$ are linearly independent over $\R$.
\par
(ii) Let $h:X\to Y$ be a surjective map such that
$h^{-1}(y)$ is a finite set for any $y\in Y$, and let
$i:X\to \R^N$ be an embedding.
Let  ${\cal X}^{\Delta}$  and $h^{\Delta}:{\cal X}^{\Delta}\to Y$ denote the space and the map
defined by
$$
{\cal X}^{\Delta}=
\big\{(y,w)\in Y\times \R^N:
w\in \sigma (i(h^{-1}(y)))
\big\}\subset Y\times \R^N,
\ h^{\Delta}(y,w)=y.
$$
The pair $({\cal X}^{\Delta},h^{\Delta})$ is called
{\it a simplicial resolution }of $(h,i)$.
If for each $y\in Y$ any $k$ points of $i(h^{-1}(y))$ span
$(k-1)$-dimensional affine subspace of $\R^N$,
$({\cal X}^{\Delta},h^{\Delta})$ is called  
{\it  a non-degenerate simplicial resolution}, otherwise the simplicial resolution is said to be {\it degenerate}.
\par
(iii)
For each $k\geq 0$, let ${\cal X}^{\Delta}_k\subset {\cal X}^{\Delta}$ be the subspace
given by 
$$
{\cal X}_k^{\Delta}=\big\{(y,\omega)\in {\cal X}^{\Delta}:
\omega\in\sigma ({\bf u}),
{\bf u}=\{u_1,\cdots ,u_l\}\subset i(h^{-1}(y)),l\leq k\big\}.
$$
We make the identification $X={\cal X}^{\Delta}_1$ by identifying the point $x$ with the pair $(h(x),i(x))$, and
obtain the following  increasing filtration:
$$
\emptyset =
{\cal X}^{\Delta}_0\subset X={\cal X}^{\Delta}_1\subset {\cal X}^{\Delta}_2\subset
\cdots \subset {\cal X}^{\Delta}_k\subset {\cal X}^{\Delta}_{k+1}\subset
\cdots \subset \bigcup_{k= 0}^{\infty}{\cal X}^{\Delta}_k={\cal X}^{\Delta}.
$$
\par\noindent{\bf Remark. }
Non-degenerate simplicial resolutions have a long history (see e.g. \cite{Va}), while degenerate ones appear to have been introduced for the first time in \cite{Mo2}.

\begin{lemma}[\cite{Mo2}, \cite{Va}]\label{lemma: simp}
Let $h:X\to Y$ be a surjective map such that, for any $y\in Y$,
$h^{-1}(y)$ is a finite set and let
$i:X\to \R^N$ be an embedding.
\par
$\I$
If $X$ and $Y$ are closed semi-algebraic spaces and the
two maps $h$, $i$ are polynomial maps, then
$h^{\Delta}:{\cal X}^{\Delta}\stackrel{\simeq}{\rightarrow}Y$
is a homotopy equivalence.
\par
$\II$
There is an embedding $j:X\to \R^M$ such that the associated simplicial resolution
$(\tilde{\cal X}^{\Delta},\tilde{h}^{\Delta})$
of $(h,j)$ is non-degenerate, and
the space $\tilde{\cal X}^{\Delta}$
is uniquely determined up to homeomorphism.
Moreover,
there is a filtration preserving homotopy equivalence
$q^{\Delta}:\tilde{\cal X}^{\Delta}\stackrel{\simeq}{\rightarrow}{\cal X}^{\Delta}$ such that $q^{\Delta}\vert X=\mbox{id}_X$.
\qed
\end{lemma}
\par\noindent{\bf Remark. }
When $h$ is not finite to one,  every simplicial resolution is degenerate. In this case it is also possible to define the associated non-degenerate resolution with properties analogous to the above (for details see  \cite{Mo2}).

\subsection{A spectral sequence for $A_d(m,n;g).$}
Fix a basepoint-preserving algebraic map $g\in\Alg_d^*(\RP^{m-1},\RP^n)$ of degree $d$ together with a representation $(g_0,\cdots,g_n)\in A_d(m-1,n)$. Recall from Section \ref{section:notation1} that we defined $A_d(m,n;g)$ as an open subspace of the affine space $A_d^*$. Let $N_d^*=\dim A_d^*=(n+1)\binom{m+d-1}{m}$. We define $\Sigma_d^*\subset A_d^*$ as the \emph{discriminant} of $A_d(m,n;g)$ in $A_d^*$:
$$
\Sigma_d^*=A_d^*\setminus A_d(m,n;g).
$$
In other words, $\Sigma_d^*$ consists of the $(n+1)$-tuples of polynomials in $A_d^*$ which have at least one nontrivial common real zero.

Our goal in this subsection is to construct a simplicial resolution associated with the discriminant and the spectral sequence for the natural filtration of this resolution. 
We start innocently enough:

\begin{lemma}\label{lemma:common-root}
If $(f_0,\cdots ,f_n)\in \Sigma_d^*$ and
${\bf x}=(x_0,\cdots ,x_m)\in \R^{m+1}$ is a non-trivial common root
of $f_0,\cdots ,f_n$, then $x_m\not= 0$.
\end{lemma}
\begin{proof}
If $x_m=0$
and
$(x_0,\cdots ,x_{m-1},0)$ was a common root of all $f_k$, then $g_k(x_0,\cdots ,x_{m-1})= f_k(x_0,\cdots ,x_{m-1},0)=0$, so all $g_k$ would have a non-trivial common root, which is a contradiction.
\end{proof}

\par\vspace{2mm}
\par\noindent{\bf Definition. }
Let  $Z_d^{*}\subset \Sigma_d^{*}\times \R^m$
denote 
{\it the tautological normalization} of 
the discriminant $\Sigma_d^*$
consisting of all pairs 
$({\bf f},{\bf x})=((f_0,\cdots ,f_n),
(x_0,\cdots ,x_{m-1}))\in \Sigma_d^{*}\times\R^m$
such that the polynomials $f_0,\cdots ,f_n$ have a non-trivial common real root
$({\bf x},1)=(x_0,\cdots ,x_{m-1},1)$.
Projection on the first factor  gives a surjective map
$\pi_d^{\p} :Z_d^{*}\to\Sigma_d^{*}.$
\par
Let $\phi_d^{*}:A_d^*\stackrel{\cong}{\rightarrow}\R^{N_d^*}$
be any fixed homeomorphism, and
let ${\rm H}_d$ be the set consisting of all
monomials $\varphi_I=z^I=z_0^{i_0}z_1^{i_1}\cdots z_m^{i_m}$
of degree $d$
($I=(i_0,i_1,\cdots ,i_m)\in\Z^{m+1}_{\geq 0}$,
$\vert I\vert =\sum_{k=0}^mi_k = d$).
Next, following \cite{Mo2} we make use of the  Veronese embedding, which will play a key role in our argument. The reason for using this embedding is that it defines a degenerate simplicial resolution, which gives rise to a spectral sequence with $E_1$ terms equal to zero outside a certain range 
(see Lemma \ref{lemma: range**} below.)
Let
$\psi_d^* :\R^m\to \R^{M_d}$
be the map given by
$\dis
\psi_d^* (x_0,\cdots ,x_{m-1})=
\Big(\varphi_I(x_0,\cdots ,x_{m-1},1)
\Big)_{\varphi_I\in {\rm H}_d},$
where $M_d:=
\binom{d+m}{m}.$
We define the embedding
$\Phi_d^* :Z_d^* \to \R^{N_d^*+M_d}$ by
$$
\Phi_d^*  ((f_0,\cdots ,f_n),{\bf x})=
(\phi_d^* (f_0,\cdots ,f_n),\psi_d^* ({\bf x})).
$$

\begin{lemma}\label{lemma: simplex*}
$\I$
If
$\{y_1,\cdots ,y_{r}\}\in C_r(\R^{m})$ is  any set of $r$ distinct points
in $\R^m$ and $r\leq d+1$, then the $r$ vectors $\{\psi_d^*(y_k):1\leq k\leq r\}$
are linearly independent
over $\R$ $($hence they
span an $(r-1)$-dimensional simplex in $\R^{M_d})$.
\par
$\II$
If $1\leq r\leq d+1$,
there is a homeomorphism
$\SZ(d)_r\setminus\SZ(d)_{r-1}\cong \tilde{\cal Z}^{\Delta}(d)_r\setminus\tilde{\cal Z}^{\Delta}(d)_{r-1}$.
\end{lemma}
\begin{proof}
Writing $y_k=(y_{0,k},\cdots ,y_{m-1,k})$ for 
$1\leq k\leq  r$, for each $i\not= j$ 
we can find a number $l$ $(0\leq l\leq m-1)$
such that $y_{l,i}\not= y_{l,j}$. By a linear change of coordinates we can ensure that $l=0$ for all $i,j$.
The assertion (i) follows form the fact that the Vandermonde matrix constructed from the powers $y_{0,i}$ is non-singular,
and (ii) easily follows from (i).
\end{proof}

\par\noindent{\bf Definition. }
Let 
$(\SZ(d),\ ^{\p}{\pi_d}^{\Delta}:\SZ(d)\to\Sigma_d^*)$ 
and
$(\tilde{\SZ}(d),\ ^{\p}\tilde{\pi}_d^{\Delta}:\SZ(d)\to\Sigma_d^*)$ 
denote the simplicial resolution of
$(\pi_d^{\p},\Phi_d^*)$ 
and the corresponding non-degenerate
simplicial resolution
with the natural increasing filtrations
$$
\begin{cases}
\dis
\SZ(d)_0=\emptyset
\subset \SZ(d)_1\subset 
\SZ(d)_2\subset \cdots
\subset 
\SZ(d)=\bigcup_{k= 0}^{\infty}\SZ (d)_k,
\\
\dis
\tilde{\SZ}(d)_0=\emptyset
\subset \tilde{\SZ}(d)_1\subset 
\tilde{\SZ}(d)_2\subset \cdots
\subset \tilde{\SZ}(d)=\bigcup_{k= 0}^{\infty}\tilde{\SZ} (d)_k.
\end{cases}
$$
By Lemma \ref{lemma: simp} the map
$^{\p}{\pi}_d^{\Delta}:
\SZ(d)\stackrel{\simeq}{\rightarrow}\Sigma_d^*$
is a homotopy equivalence, and it
extends to  a homotopy equivalence
$^{\p}{\pi_{d+}^{\Delta}}:\SZ(d)_+\stackrel{\simeq}{\rightarrow}{\Sigma_{d+}^*},$
where $X_+$ denotes the one-point compactification of a
locally compact space $X$.
\par
Since
${\SZ (d)_r}_+/{\SZ (d)_{r-1}}_+
\cong (\SZ (d)_r\setminus \SZ (d)_{r-1})_+$,
we have the spectral sequence 
$\big\{E_t^{r,s}(d),
d_t:E_t^{r,s}(d)\to E_t^{r+t,s+1-t}(d)
\big\}
\Rightarrow
H^{r+s}_c(\Sigma_d^*,\Z),$
where
$H_c^k(X,\Z)$ denotes the cohomology group with compact supports given by 
$H_c^k(X,\Z):= H^k(X_+,\Z)$ and
$E_1^{r,s}(d):=H^{r+s}_c(\SZ(d)_r\setminus\SZ(d)_{r-1},\Z)$.
\par

It follows from the Alexander duality that there is a natural
isomorphism
\begin{equation}\label{Al}
H_k(A_d(m,n;g),\Z)\cong
H_c^{N_d^*-k-1}(\Sigma_d^*,\Z)
\quad
\mbox{for }1\leq k\leq N_d^*-2.
\end{equation}
Using (\ref{Al}) and
reindexing we obtain a
spectral sequence
\begin{eqnarray}\label{SS}
&&\big\{\E^t_{r,s}(d), \tilde{d}^t:\E^t_{r,s}(d)\to \E^t_{r+t,s+t-1}(d)\big\}
\Rightarrow H_{s-r}(A_d(m,n;g),\Z)
\end{eqnarray}
if $s-r\leq N_d^*-2$,
where
$\E^1_{r,s}(d)=
H^{N_d^*+r-s-1}_c(\SZ(d)_r\setminus\SZ(d)_{r-1},\Z).$

\begin{lemma}\label{lemma: vector bundle*}
If $1\leq r\leq \lfloor\frac{d+1}{2}\rfloor$,
$\SZ(d)_r\setminus\SZ(d)_{r-1}$
is homeomorphic to the total space of a real
vector bundle $\xi_{d,r}$ over $C_r(\R^m)$ with rank 
$l_{d,r}^*:=N_d^*-nr-1$.
\end{lemma}
\begin{proof}
For ${\bf y}=\{y_1,\cdots ,y_r\}\in C_r(\R^m)$, let $\sigma_{\psi} ({\bf y})$
denote the convex hull spanned by $\{\psi_d^* (y_1),\cdots ,\psi_d^* (y_r)\}$,
and let $\mbox{int}(\sigma_{\psi}({\bf y}))$ be its interior.
By Lemma \ref{lemma: simplex*} , if
$1\leq r\leq \lfloor\frac{d+1}{2}\rfloor$
and $\sigma_{1}$ and $\sigma_{2}$ are
any two $(r-1)$-simplices in $\R^{M_d}$ with
$\sigma_1\not=\sigma_2$ whose vertices are
in the image of $\psi_d^*$, then either
$\sigma_{1}\cap\sigma_{2}=\emptyset$ or
$\sigma_{1}\cap \sigma_{2}$ is their common face of
lower dimension.
Therefore, if
$1\leq r\leq \lfloor \frac{d+1}{2}\rfloor$
and $y\in \R^{M_d}$
is contained in the interior of some $(r-1)$-simplex $\sigma$
whose vertices are in the image of $\psi_d^*$,
all vertices of  this $\sigma$ are uniquely determined up to order.
Hence, if $1\leq r\leq \lfloor \frac{d+1}{2}\rfloor$,
one can define a map
$\pi:\SZ(d)_r\setminus\SZ(d)_{r-1}\to C_r(\R^m)$
by 
$$
((f_0,\cdots ,f_n),y)\mapsto
\left\{
\begin{array}{c}
\mbox{ }
\\
\{y_1,\cdots ,y_r\}
\\
\mbox{ }
\end{array}\right|
\left.
\begin{array}{l}
(y_k,1) \in\R^{m+1} \mbox{ is a common root}
\\
\mbox{ of }(f_0,\cdots ,f_n)
\mbox{ for }1\leq k\leq r,
\\
y\in 
\mbox{int}(\sigma_{\psi}(\{y_1,\cdots ,y_r\})),
\\
y_i\not=y_j\mbox{ if }i\not=j
\end{array}\right\}.
$$
In general,
the condition that a polynomial in $B_k=\{g_k+z_mh:h\in {\cal H}_{d-1,m}\}$ 
vanishes at a given non-zero point 
gives one linear condition on its coefficients, and 
determines an affine hyperplane in $B_k$.
By Lemma \ref{lemma: simplex*}, if $r\leq d+1$, then
the condition that a polynomial in $B_k$ 
vanishes at $r$ distinct  
points ${\bf x}=\{x_k\}_{k=1}^r$
produces exactly $r$ independent conditions on its coefficients
if and only if the corresponding convex hull $\sigma_{\psi} ({\bf x})$
is an $(r-1)$-simplex.
Hence, if $1\leq r\leq d$, the space of polynomials in $B_k$
which
vanish at $r$ distinct points is the intersection of
$r$ affine hyperplanes in general position and thus has codimension $r$
in $B_k$.
\par
Thus, 
the fiber of $\pi$
is homeomorphic to the product of open $(r-1)$-simplex with
the real vector space of dimension
$(n+1)\big(\binom{d+m-1}{m}-r\big)=N_d^*-(n+1)r.$
One can also easily see that the map $\pi$ is the projection map of a locally trivial fiber bundle.
Hence, $\pi$ is a real vector bundle over $C_r(\R^m)$ with
the rank $l^*_{d,r}$, where
$l_{d,r}^*:=N_d^*-(n+1)r+r-1=N_d^*-nr-1.$
\end{proof}

\begin{lemma}\label{lemma: range**}
All non-zero entries of $\E^1_{r,s}(d)$ are
 situated in the range
$s\geq r\big(n+1-m\big)$ with
$r\geq 0$.
\end{lemma}
\begin{proof}
An element of $\SZ(d)_r\setminus\SZ(d)_{r-1}$ is determined by the following $3$ pieces of data:

\begin{itemize}
\item
$r$ distinct points $\{y_1,\cdots ,y_r\}\in C_r(\R^m)$, whose 
corresponding images of $\psi_d$,
${\cal F}=\{\psi_d^*(y_1),\cdots ,\psi_d^*(y_r)\}$  are in a general position.
\item
A point in the convex hull spanned by these points in ${\cal F}$.
\item
A collection of $(n+1)$-polynomials $f_0,\dots, f_n$, which all vanish at the $r$ points $\{(y_1,1)\cdots ,(y_r,1)\}.$
\end{itemize}
Note that the points ${\cal F}=\{\psi_d^*(y_1),\cdots ,\psi_d^*(y_r)\}$ must be in a general position (otherwise our point would belong to $\SZ(d)_{r-1}$).
Hence, the corresponding vectors are linearly independent (note that the last coordinate is always $1$).
The coordinates of these vectors are precisely the monomials appearing in the equations $f_i(y_j,1)=0$, $(0\leq i\leq n,1\leq j\leq r)$, so for  each $0\leq i \leq n$, the system of equations in the coefficients of the polynomial  $f_i$, arising from the conditions that it vanishes at the points $\{(y_1,1), \dots, (y_r,1)\}$, consists of linearly independent equations. Thus, we obtain the inequality
\begin{eqnarray*}
\dim \big(\SZ(d)_r\setminus\SZ(d)_{r-1}\big)
&\leq & rm + (r-1) + \big(N_d^*-r(n+1)\big) 
\\
&=& N_d^*-r(n-m)-1.
\end{eqnarray*}
Since
$N_d^*-r(n-m)-1\geq N_d^*+r-s-1$
if and only if
$s\geq r(n+1-m),$
the assertion follows from (\ref{SS}).
\end{proof}
\begin{lemma}\label{lemma: range*}
If $1\leq r\leq \lfloor \frac{d+1}{2}\rfloor$, there is a natural isomorphism
$$
\E^1_{r,s}(d)\cong
H_{s-(n-m+1)r}(C_r(\R^m),(\pm \Z)^{\otimes (n-m)}).
$$
Here the meaning of  $(\pm \Z)^{\otimes (n-m)}$  is the same as in {\rm \cite{Va}}.
\end{lemma}
\par\noindent{\it Proof. }
Suppose that $1\leq r\leq \lfloor \frac{d+1}{2}\rfloor$.
Then
by Lemma \ref{lemma: vector bundle*}, there is a
homeomorphism
$(\SZ(d)_r\setminus\SZ(d)_{r-1})_+\cong T(\xi_{d,r}),$
where $T(\xi)$ denotes the Thom complex of the vector bundle
$\xi$.
Since $N_d^*+r-s-1-l_{d,r}^*=(n+1)r-s$
and $rm-\{(n+1)r-s\}=s-(n-m+1)r$,
by using the Thom isomorphism and Poincar\'e duality,
we obtain a natural isomorphism 
$$
\E^1_{r,s}(d)
\cong H^{N_d^*+r-s-1}(T(\xi_{d,r}),\Z)
\cong
H_{s-(n-m+1)r}(C_r(\R^m), (\pm \Z)^{\otimes (n-m)}).
\qed
$$
\par\vspace{2mm}\par

\subsection{Comparison with Vassiliev's spectral sequence for $\Omega^mS^n.$}

Recall the spectral sequence constructed by V. Vassiliev
\cite[page 109--115]{Va}.
From now on, we will assume that $m<n$ and $X$ is 
a finite $m$-dimensional simplicial complex $C^{\infty}$-imbedded in $\R^L$.
Considering $S^n$ and $X$ as subspaces $S^n\subset \R^{n+1},  X\subset \R^L$, 
we identify $\Map (X,S^n)$ with the space
$\Map (X,\R^{n+1}\setminus \{{\bf 0}_{n+1}\})$.
We also choose and fix a map
$\varphi :X\to \R^{n+1}\setminus\{{\bf 0}_{n+1}\}$.
Observe that $\Map (X,\R^{n+1})$ is a linear space
and consider
the complements
${\frak A}_m^n(X)=\Map (X,\R^{n+1})\setminus \Map (X,S^n)$ and
$\tilde{\frak A}_m^n(X)=\Map^* (X,\R^{n+1})\setminus \Map^*(X,S^n).$
\par
Note that  ${\frak A}_m^n(X)$ consists of all continuous maps
$f:X\to \R^{n+1}$ passing through ${\bf 0}_{n+1}.$
We will denote by $\Theta^k_{\varphi}(X)\subset \Map (X,\R^{n+1})$ the subspace
consisting of all maps $f$ of the forms $f=\varphi +p$,
where $p$ is the restriction to $X$ of a polynomial map
$\R^L\to \R^{n+1}$ of degree $\leq k$.
Let
$\Theta^k_X\subset \Theta^k_{\varphi}(X)$ denote the subspace
consisting of all $f\in \Theta_{\varphi}^k(X)$
passing through ${\bf 0}_{n+1}.$ 
In \cite[page 111-112]{Va} Vassiliev uses the space 
$\Theta^k(X)$ as a finite dimensional approximation of ${\frak A}^n_m(X)$.\footnote
{Note that the proof 
 of this fact given by Vassiliev, makes use of the Stone-Weierstrass theorem, 
 so, although we are now not using the stable result of Section \ref{section:stable}, something like it is also implicitly  involved here. }
\par
Let $\tilde{\Theta}^k_X$ denote the subspace of $\Theta^k_X$ consisting of all maps 
$f\in \Theta^k_X$ which preserve
the base points.
By a variation of the preceding argument, Vassiliev also shows that  
$\tilde{\Theta}^k_X$ can be used as a finite dimensional approximation of  
$\tilde{\frak A}_m^n(X)$ \cite[page 112]{Va}.
\par
Let ${\cal X}_k\subset \tilde{\Theta}^k_X\times \R^{L}$ denote the 
subspace consisting of all pairs 
$(f,\alpha)\in\tilde{\Theta}^k_X \times \R^{L}$
such that $f(\alpha )={\bf 0}_{n+1}$, and
let $p_k:{\cal X}_k\to \tilde{\Theta}^k_X$ be the projection onto the first factor.
Then,  by making use of (non-degenerate) simplicial resolutions of the surjective maps 
$\{p_k:k\geq 1\}$,
one can construct a geometric resolution 
$\{\tilde{\frak A}_m^n(X)\}$ of $\tilde{\frak A}_m^n,$ 
whose cohomology is naturally isomorphic to the homology of $\Map^*(X,S^n)$. 
From the natural filtration 
$\dis F_1\subset F_2\subset F_3\subset 
\cdots \subset \bigcup_{k=1}^{\infty}F_k=\{\tilde{\frak A}_m^n(X)\},$
we obtain the associated spectral sequence:
\begin{equation}\label{SSS}
\{E^t_{r,s},d^t:
E^t_{r,s}\to
E^t_{r+t,s+t-1}\}
\Rightarrow
H_{s-r}(\Map^*(X,S^n),\Z).
\end{equation}
The following result follows easily from
\cite[Theorem 2 (page 112) and (32) (page 114)]{Va}. 
\begin{lemma}[\cite{Va}]
\label{lemma: Va}
Let $2\leq m<n$ be integers and let
$X$ be a finite $m$-dimensional simplicial complex
with a fixed base point $x_0\in X$.
\begin{enumerate}
\item[$\I$]
If $r\geq 1$,
$E^1_{r,s}=
H_{s-(n-m+1)r}(C_r(X\setminus\{x_0\}),(\pm \Z)^{\otimes (n-m)}).$
\item[$\II$]
If $r<0$ or $s<0$ or $s<(n-m+1)r$, then $E^1_{r,s}=0$.
\item[$\III$]
If $X=S^m,$ then,
for any $t\geq 1$,
$d^t=0:
E^t_{r,s}\to
E^t_{r+t,s+t-1}$ for all $(r,s)$,
and $E^1_{r,s}=E^{\infty}_{r,s}$.
Moreover,
for any $k\geq 1$, the extension problem for
$Gr (H_k(\Omega^mS^n,\Z))=
\bigoplus_{r=1}^{\infty}E^{\infty}_{r,r+k}$
is trivial and 
there is an isomorphism
$$
H_{k}(\Omega^mS^n,\Z)\cong \bigoplus_{r=1}^{\infty}E^1_{r,k+r}=
\bigoplus_{r=1}^{\infty}
H_{k-(n-m)r}(C_r(\R^m),(\pm \Z)^{\otimes (n-m)}).
\qed$$
\end{enumerate}
\end{lemma}

Although we are really interested in the inclusion map $i_d^{\p}$ of Theorem \ref{thm: II}, we will first define another map $i_d^{\p\p}$, for which the analogous result is easier to prove. We will then deduce the result for $i_d^{\p}$ from the one for $i_d^{\p\p}$.
\par\vspace{2mm}
\par\noindent{\bf Definition. }
Let $i_d^{\p\p}:A_d(m,n;g)\to \Omega^m (\R^{n+1}\setminus \{{\bf 0}\})\simeq \Omega^mS^n$ by given by
$$
i_d^{\p\p}({\bf f})({\bf x})=
\big(f_0(x_0,\cdots ,x_m),\cdots\cdots , f_n(x_0,\cdots,x_m)\big)
$$
for $({\bf f},{\bf x})=((f_0,\cdots ,f_n),(x_0,\cdots ,x_m))\in A_d(m,n;g)\times S^m$.
\par\vspace{2mm}\par
We will prove the analogue of Theorem \ref{thm: II} 
for the map $i_d^{\p\p}$
by applying the spectral sequence
(\ref{SSS}) to
the case $X=S^m$.
\begin{thm}\label{thm: II**}
Let $2\leq m<n$ be integers and
$g\in\Alg_d^*(\RP^{m-1},\RP^n)$ be a fixed algebraic map.
The map
$i_d^{\p\p}:A_d(m,n;g)\to  \Omega^mS^n$
is a homotopy equivalence through dimension $D(d;m,n)$ if $m+2\leq n$
and a homology equivalence through dimension $D(d;m,n)$ if $m+1=n$. 
\end{thm}

\begin{proof}
From now on, we identify $\Omega^mS^n=\Omega^m(\R^{n+1}\setminus \{{\bf 0}\})$ and
consider the spectral sequence (\ref{SSS}) for $X=S^m$.
Note that the image of the  map $i_d^{\p\p}$ lies 
in a space of mappings that arise from restrictions of polynomial mappings
$\R^m\to \R^n$.
Since  $\tilde{\cal Z}^{\Delta}(d)$ is a non-degenerate simplicial resolution,  the map
 $i_d^{\p\p}$ naturally extends to a filtration
preserving map 
$^{\p}{\tilde{\pi}}: \tilde{\cal Z}^{\Delta} (d)\to \{\tilde{\frak A}_m^n(S^n)\}$
between resolutions.
By Lemma \ref{lemma: simp} there is a filtration
preserving homotopy equivalence 
$q^{\Delta} :\tilde{\cal Z}^{\Delta} (d)\stackrel{\simeq}{\rightarrow}\SZ (d)$.
Then the filtration preserving maps
$\begin{CD}\dis
\SZ (d) @<q^{\Delta}<\simeq< \tilde{\cal Z}^{\Delta} (d) 
@>^{\p}{\tilde{\pi}}>> \{\tilde{\frak A}_m^n(S^n)\}
\end{CD}$
 induce a homomorphism of spectral sequences
$\{\tilde{\theta}^t_{r,s}:\E^t_{r,s}(d)\to E^t_{r,s}\},$
where $\{\E^t_{r,s}(d),\tilde{d}^t\}\Rightarrow H_{s-r}(A_d(m,n;g),\Z)$ and
$\{E^t_{r,s},d^t\}\Rightarrow H_{s-r}(\Omega^mS^n,\Z).$
\par
Note now that,
by Lemma \ref{lemma: simplex*},
Lemma \ref{lemma: range*}, 
Lemma \ref{lemma: Va} and
the naturality of Thom isomorphism,
for $r\leq \lfloor \frac{d+1}{2}\rfloor$
 there is a
commutative diagram
\begin{equation}\label{Thom2}
\begin{CD}
\E^1_{r,s}(d) 
@>T>\cong> H_{s-r(n-m+1)}(C_r(\R^m),(\pm \Z)^{\otimes (n-m)})
\\
@V\tilde{\theta}^1_{r,s}VV  \Vert @.
\\
E^1_{r,s} 
@>T>\cong> H_{s-r(n-m+1)}(C_r(\R^m),(\pm \Z)^{\otimes (n-m)} )
\end{CD}
\end{equation}
Hence, if $r\leq \lfloor \frac{d+1}{2}\rfloor$,
$\tilde{\theta}^1_{r,s}:\E^1_{r,s}(d)\stackrel{\cong}{\rightarrow}
E^1_{r,s}$ and thus so is
$\tilde{\theta}^{\infty}_{r,s}:\E^{\infty}_{r,s}(d)\stackrel{\cong}{\rightarrow}
E^{\infty}_{r,s}$.
\par
Next, consider the number
$$
D_{min}=\min\{
N\vert \ N\geq s-r,\ 
s\geq (n+1-m)r,\ 
1\leq r<\lfloor \frac{d+1}{2}\rfloor +1
\}.
$$
Clearly  $D_{min}$ is the
largest integer $N$ which satisfies the inequality
$(n+1-m)r> r+N$
for $r=\lfloor\frac{d+1}{2}\rfloor +1$, hence
\begin{equation}\label{Dmin}
D_{min}=
(n-m)(\lfloor\frac{d+1}{2}\rfloor +1)-1=D(d;m,n).
\end{equation}
We note that, for dimensional reasons,
$\tilde{\theta}^{\infty}_{r,s}:
\E^{\infty}_{r,s}(d)\stackrel{\cong}{\rightarrow} 
E^{\infty}_{r,s}$ is always
an isomorphism when
$r\leq \lfloor \frac{d+1}{2}\rfloor$ and $s-r\leq D(d;m,n)$.
Remark also that by Lemma \ref{lemma: range**},
$\tilde{E}^1_{r,s}(d)=E^1_{r,s}=0$ when
$s-r\leq D(d;m,n)$ and $r>\lfloor\frac{d+1}{2}\rfloor$.
Hence, we have:
\begin{enumerate}
\item[(\ref{Dmin}.1)]
If  $s\leq r+D(d;m,n)$, 
$\tilde{\theta}^{\infty}_{r,s}:
\E^{\infty}_{r,s}(d)\stackrel{\cong}{\rightarrow} 
E^{\infty}_{r,s}$ is always
an isomorphism.
\end{enumerate}
We now see that
$i_d^{\p\p}$ 
is a homology equivalence through dimension $D(d;m,n)$.
If $m+2\leq n$ then, since $A_d(m,n;g)$ and $\Omega^mS^n$ are simply connected,
$i_d^{\p\p}$ is a homotopy equivalence through dimension $D(d;m,n)$.
Thus, (i) is proved.
\end{proof}

\begin{cor}\label{cor: II*}
Let $2\leq m<n$ be integers,
$g\in\Alg_d^*(\RP^{m-1},\RP^n)$,
and let $\Bbb F =\Z/p$ $($p: prime$)$ or $\Bbb F =\Bbb Q$.
Then the map $i_d^{\p\p}:A_d(m,n;g)\to \Omega^mS^n$ induces an isomorphism on
the homology group $H_k(\cdot,\Bbb F)$ for any $1\leq k\leq D(d;m,n)$.
\end{cor}
\begin{proof}
If we
apply the argument in the  proof of Theorem \ref{thm: II**}  above, with the homology groups $H_k(\ ,\Z)$ and $H_k(\ ,(\pm \Z)^{\otimes k})$ replaced by $H_k(\ ,\Bbb F)$
and $H_k(\ ,(\pm \Bbb F)^{\otimes k})$, the assertion follows.
\end{proof}

\subsection{Proof of Theorem \ref{thm: II}.}

\paragraph{Proof of $\I$.}
Fix a map $g$ as in the assumptions of the theorem. Let $\gamma_k:S^k\to \RP^k$ be the usual double covering map. Denote by
$i^{\p}:F(m,n;g)
\stackrel{\subset}{\rightarrow} 
\Map_{[d]_2}^*(\RP^m,\RP^n)$ the inclusion, and let 
$\gamma_m^{\#}:\Map^*_{\epsilon}(\RP^m,\RP^n)\to \Omega^m\RP^n$
be the map given by
$\gamma_m^{\#}(h)=h\circ \gamma_m$.
Then the following diagram is commutative.
\begin{equation}\label{gamma}
\begin{CD}
A_d(m,n;g) @>\Psi_d^{\p}>> \Alg_d^*(m,n;g) @>i_{d,\R}^{\p}>\subset> F(m,n;g)
\\
@V{i_d^{\p\p}}VV @. @V{i^{\p}}V{\cap}V
\\
\Omega^mS^n @>\Omega^m\gamma_n>\simeq> \Omega^m\RP^n @<\gamma_m^{\#}<< \Map^*_{[d]_2}(\RP^m,\RP^n)
\end{CD}
\end{equation}

Since the spaces $A_d(m,n;g)$ and $F(m,n;g)\simeq\Omega^mS^n$
are simply connected if $m+2\leq n$,
it suffices to show that $i_{d,\R}^{\p}\circ \Psi_d^{\p}$ is a
homology equivalence through dimension $D(d;m,n)$.
However, as $\Omega^m\gamma_n$ is a homotopy equivalence,
applying Theorem \ref{thm: II**} and the diagram (\ref{gamma}) we see that
it suffices to show that the map
$\gamma_m^{\#}\circ i^{\p}$ is a homology
equivalence through dimension $D(d;m,n)$.
\par
Let $\Bbb F =\Z/p$ $(p$: prime) or $\Bbb F =\Bbb Q$. Consider the induced homomorphism
$(\gamma_m^{\#}\circ i^{\p})_*=H_k(\gamma_m^{\#}\circ i^{\p},\Bbb F)
:H_k(F(m,n;g),\Bbb F)\to H_k(\Omega^m\RP^n,\Bbb F)$.
By Corollary \ref{cor: II*} and (\ref{gamma}),
$(\gamma_m^{\#}\circ i^{\p})_*$ is an epimorphism for any $1\leq k\leq D(d;m,n)$.
However, since there is a homotopy equivalence
$F(m,n;g)\simeq \Omega^m\RP^n$,
$\dim_{\Bbb F}H_k(F(m,n;g),\Bbb F)=\dim_{\Bbb F}H_k(\Omega^m\RP^n,\Bbb F)<\infty$
for any $k$.
Hence,
$H_k(\gamma_m^{\#}\circ i^{\p},\Bbb F )$ is an isomorphism for any $1\leq k\leq D(d;m,n)$.
Hence,  by the Universal Coefficient Theorem,
$\gamma_m^{\#}\circ i^{\p}$ is a homology equivalence through
dimension $D(d;m,n)$.

\paragraph{Proof of $\II$.}
Since
$d^t=0$ for any $t\geq 1$, from
the equality
$d^t\circ \tilde{\theta}^t_{r,s}
=\tilde{\theta}^t_{r+t,s+t-1}\circ \tilde{d}^t$
and some diagram chasing,
we obtain $\E^1_{r,s}(d)=\E^{\infty}_{r,s}(d)$ for all $r\leq \lfloor \frac{d+1}{2}\rfloor$.
Moreover, since
the extension problem for
$Gr (H_k(\Omega^mS^n,\Z))=\bigoplus_{r=1}^{\infty}E^{\infty}_{r,k+r}
=\bigoplus_{r=1}^{\infty}E^{1}_{r,k+r}$
is trivial, we deduce from(\ref{Dmin}.1) that the associated graded group
$Gr (H_k(A_d(m,n;g),\Z))=\bigoplus_{r=1}^{\infty}\E^{\infty}_{r,k+r}(d)$
is also trivial until the $\lfloor\frac{d+1}{2}\rfloor$-th term of the filtration.
Hence,  $H_k(A_d(m,n;g),\Z)$ contains the subgroup
$$
\bigoplus_{r=1}^{\lfloor \frac{d+1}{2}\rfloor}\E^1_{r,k+r}(d)
=
\bigoplus_{r=1}^{\lfloor \frac{d+1}{2}\rfloor}\E^{\infty}_{r,k+r}(d)
\cong
\bigoplus_{r=1}^{\lfloor \frac{d+1}{2}\rfloor}
H_{k-r(n-m)}(C_r(\R^m),(\pm \Z)^{\otimes (n-m)})
$$
as a direct summand, which proves the assertion (ii).
\qed

\subsection{Proof of Theorem \ref{thm: III}.}

This proof is almost identical to that in \cite[Section 6]{Mo2}.
We first prove, by induction on $m$, that the map $i_{d}: A_d(m,n)   \to \Map_{[d]_2}^*(\RP^m,\RP^n)$ is a homology equivalence through dimension $D(d;m,n)$. 
If $m=1$, the assertion follows from Theorem \ref{thm: A2}. 
Suppose that the assertion is true for $m-1$ for some $m\geq 2$, and 
consider the fiberwise map 
$R: A_d(m,n) \to A_d(m-1,n)$
given by setting $z_{m}=0$. 
Restriction to the hyperplane $z_{m}=0$ gives a map
$
r :\Map_{[d]_2}^*(\RP^m,\RP^n)\to \Map_{[d]_2}^*(\RP^{m-1},\RP^n)
$,
which is a Serre fibration with fiber $\Omega^mS^n$.  We have the following diagram
\begin{equation*}
\begin{CD}
A_d(m,n;\cdot) @>>> F(m,n;\cdot) @>>{=}>  F(m,n;\cdot)
\\
@VV{\cap}V @VV{\cap}V @VV{\cap}V
\\
A_d(m,n) @>i^{\p\p}>{\subset}> 
M_d^*(\RP^m,\RP^n) @>j^{\p\p}>{\subset}>
 \Map_{[d]_2}^*(\RP^m,\RP^n)
\\
@V{R}VV @V{r^{\p}}VV @V{r}VV
\\
A_d(m-1,n)  @>>{=}> 
A_d(m-1,n)   @>i_d>{\subset}> \Map_{[d]_2}^*(\RP^{m-1},\RP^n)
\end{CD}
\end{equation*}
in which the rightmost square is a pullback and $\M_d^*(\RP^m,\RP^n)$ is the space of continuous maps whose restriction to $\RP^{m-1}$ is algebraic of degree $d$. 
Since the base spaces of the fibrations   $r$ and $r^{\p}$ are simply connected, comparison of their spectral sequences shows by induction that the inclusion 
$j^{\p\p}:\M_d^*(\RP^m,\RP^n)\to \Map_{[d]_2}^*(\RP^m,\RP^n)$ is a homology equivalence through dimension $D(d;m-1,n)>D(d;m,n)$.


It suffices to show that $i^{\p\p}:A_d(m,n)  \to M_d^*(\RP^m,\RP^n)$ is a homology equivalence through dimension $D(d;m,n)$.

By Theorem \ref{thm: II}, $i^{\p\p}$ induces isomorphisms  through dimension $D(d;m,n)$ on all fibers.  Now we use an argument analogous to the one in \cite[Sect.6]{Mo2}. Our case is, in fact, somewhat easier since the spaces $A_d(m,n)$ and $A_d(m-1,n)$ are complements to closed subvarieties in affine spaces. As all the results from the theory of stratifying spaces from \cite{GM} are valid in the real algebraic case, the rest of the argument applies without any changes.\footnote{We note in passing that in \cite{Mo2} the word \lq stratum\rq\ is used for what is normally the closure of a \lq stratum\rq\ .} The homotopy statement for $m+2\leq n$ follows by Whitehead's theorem and simple-connectivity of $A_d(m,n)$ and $\Map_{[d]_2}(\RP^m,\RP^n)$.

It remains to prove that
$j_d:\tilde{A}_d(m,n)\to \Map_{[d]_2}(\RP^m,\RP^n)$
is a homology (resp. homotopy) equivalence through dimension $D(d;m,n)$. Consider the evaluation map 
$\tilde{ev}:\tilde{A}_d(m,n)\to\RP^n$ given by
$\tilde{ev}
([f_0:\cdots :f_n])=[f_0({\bf e}_m):\cdots :f_n({\bf e}_m)]$,
where ${\bf e}_k:=(1,0,\cdots ,0)\in\R^{k+1}$. This is a fibration (see eg. \cite[(2.1)]{CMN}) and $\tilde{ev}^{-1}({\bf e}_n)$ is the space of $\R^*$-equivalence classes of $(n+1)$-tuples of polynomials $[f_0:\cdots:f_n]$ such that the coefficient at $z_0^d$ equals $0$ in all $f_k$, $k\geq 1$. Every such class has a unique representative such that the coefficient at $z_0^d$ in $f_0$ is $1$, so the fiber is easily identifiable with $A_d(m,n)$.

Therefore we have a commutative diagram of fibration sequences
$$
\begin{CD}
A_d(m,n) @>>> \tilde{A}_d(m,n) @>\tilde{ev}>> \RP^n
\\
@V{i_d}VV
@V{j_d}VV \Vert @.
\\
\Map^*_{[d]_2}(\RP^m,\RP^n) @>>> 
\Map_{[d]_2}(\RP^m,\RP^n) @>ev>> \RP^n
\end{CD}
$$
which translates the statements about $i_d$ into statements about $j_d$.
\qed
\par\vspace{2mm}\par
\noindent{\bf Remark. }
The above argument depends on rather difficult results from the theory of stratifying spaces. In an early version of this article we gave more direct proof of Theorem \ref{thm: III} which, however, works only for spaces of maps of even degrees.
\par 
Indeed, for $d\equiv 0$ $\mo$ is an even integer,
consider the map
$j_d^{\p}:A_d(m,n)\to \Map^*(\RP^m,\R^{n+1}\setminus
\{{\bf 0}_{n+1}\})\simeq \Map^*(\RP^m,S^n)$ given by
$$
j^{\p}_d(f_0,\cdots ,f_n)([x_0:\cdots :x_m])=
\Big(\frac{f_0(x_0,\cdots ,x_m)}{\sum_{k=0}^mx_k^d},
\cdots ,\frac{f_n(x_0,\cdots ,x_m)}{\sum_{k=0}^mx_k^d}
\Big).
$$
Then we can prove:

\begin{thm}\label{thm: old}
If $2\leq m <n$ and 
$d\equiv 0$ $\mo$ are positive integers, 
 the map
$j_d^{\p}:A_d(m,n)\to \Map^*(\RP^m,S^n)$
is a homotopy equivalence through dimension $D(d;m,n)$ if
$m+2\leq n$ and a homology equivalence through dimension
$D(d;m,n)$ if $m+1=n$.
\end{thm}

  
The proof is similar to the proof of Theorem \ref{thm: II}; the key difference being that the Veronese map 
$\psi_d^* :\R^m\to \R^{M_d}$ is replaced by a map $\psi_d :\RP^m\to \R^{M_d}$
given by
$$
\psi_d ([x_0:\cdots :x_m])=
\Big[\frac{\varphi_I(x_0,\cdots ,x_m)}{\sum_{k=0}^{m}x_k^{d}}
\Big]_{\varphi_I\in {\rm H}_d}
$$
The maps $j_d^{\p}$ and $\psi_d$ are well defined only for even $d$ which is why the argument only works in this case.
Because  the map
$\gamma_{n\#}:\Map^*(\RP^m,S^n)
\stackrel{\simeq}{\rightarrow}
\Map_0^*(\RP^m,\RP^n)$
is a homotopy equivalence,
 we can  prove Theorem 
\ref{thm: III} by using  Theorem \ref{thm: old}
when $d\equiv 0$ $\mo$.
\section{Proof of Theorem \ref{thm: A4}.}\label{section 5}

Let $\tilde{P}\subset ({\cal H}_{d,m}^{n+1}\setminus
\{{\bf 0}\})/\R^*$ 
denote the subspace consisting of those $(n+1)$-tuples 
$[f_0:\cdots :f_n]$ for which the coefficient at $z_0^d$ equals $0$ in $f_k$ if $k\not= 0$. 
Note that $\tilde{P}$ is homeomorphic with $\RP^N$ for some $N$ and that $A_d(m,n)$ is homeomorphic to 
the subspace $\tilde{P}_d$ of $\tilde{P}$ consisting of tuples with no common non-trivial real zero.
From now on, we identify $A_d(m,n)=\tilde{P}_d$.

For what follows, we need to fix a metric on $\CP^m$. Note that there is a canonical embedding $\RP^m\subset \CP^m$. For each $\epsilon >0$, let $V_{\epsilon}\subset \tilde{P}$
denote the subspace consisting of all $[f_0:\cdots :f_n]\in \tilde{P}$
which satisfy the following condition:
\begin{itemize}
\item[(*)] there exists an open ball $B$ of radius $\epsilon$ in $\CP^m$, centered at a point of $\RP^m\subset \CP^m$, such that each $f_k$ has a zero in $B$.
\end{itemize}
Note that each $V_\epsilon$ is an open subset of $\tilde{P}$ and that $V_{\epsilon_1}\subset V_{\epsilon_2}$ if $0<\epsilon_1<\epsilon_2$. Moreover
$$
\bigcap_{\epsilon>0}V_{\epsilon}=\tilde{P}\setminus A_d(m,n).
$$
The map $\Psi_d :A_d(m,n)\to \Alg_d^*(\RP^m,\RP^n)$ restricts to maps:
$$
\Psi_d^\epsilon=\Psi_d\vert_{\tilde{P}\setminus V_{\epsilon}}:
\tilde{P}\setminus V_{\epsilon}\to
\Psi_d (\tilde{P}\setminus V_{\epsilon})\subset \Alg_d^*(\RP^m,\RP^n).
$$
Since $\tilde{P}\setminus V_{\epsilon}$ is a closed, and hence compact subset of the projective space $\tilde{P}$, 
the image of $\Psi_d^\epsilon$ is also compact. 
Therefore the restricted map $\Psi_d^\epsilon$ is a closed map of Hausdorff spaces.
\par
Note that
any fiber $F_{\epsilon}$ of $\Psi_d^\epsilon$ is homeomorphic to the space of positive real homogenous polynomials of fixed even degree with
 leading coefficient  equal to $1$ 
and no zeros 
in the $\epsilon$-neighborhood of $\RP^m$ in $\CP^m$.
Clearly 
$F_{\epsilon^{\p}}\subset F_{\epsilon}$ if
$0<\epsilon <\epsilon^{\p}$.
We also have the following:
\begin{lemma}\label{lemma: F}
For $0<\epsilon <\epsilon^{\p}$, the inclusion
$F_{\epsilon^{\p}}\subset F_{\epsilon}$ is a deformation retract.
\end{lemma}
\begin{proof}[Proof of Lemma \ref{lemma: F}]
For $f=f(z_0,\cdots ,z_m)\in F_{\epsilon}$ and $0\leq t<1$,
let $\phi_t (f)$ denote the positive
homogenous polynomial given by
$$
\phi_t(f)=f\Big(z_0,\frac{z_1}{(1+\tan \frac{\pi t}{2})},
\frac{z_2}{(1+\tan \frac{\pi t}{2})^2},
\cdots ,
\frac{z_m}{(1+\tan \frac{\pi t}{2})^m}
\Big).
$$
Observe that
if ${\bf x}=[x_0:\cdots :x_m]\in \CP^m$ is a root of $f$,
 $\phi_t (f)$ has the corresponding root
${\bf x}_t=[x_0:x_1(1+\tan \frac{\pi t}{2}):
x_2(1+\tan \frac{\pi t}{2})^2:\cdots :x_m(1+\tan \frac{\pi t}{2})^m]$.
Thus the distance  $D(t)$ between ${\bf x}_t$ and $\RP^m$ in $\CP^m$
is an increasing continuous function and it goes to $\infty$
if $t\to 1-0$.
Hence $\phi_t(f)$ has no zeros in the $\epsilon^{\p}$-neighborhood
of $\RP^m$ in $\CP^m$ for some $0<t<1$ and the assertion follows.
\end{proof}
By applying Lemma \ref{lemma: F} we  see that
the inclusion 
$\dis F_{\epsilon}\stackrel{\simeq}{\rightarrow} 
F:=\bigcup_{\delta >0}F_{\delta}$ is a
homotopy equivalence.
Moreover, since 
$F$
is the space consisting of all 
positive homogenous polynomials of the fixed
even degree with the leading coefficient equal to $1$,
it is convex, hence 
contractible, and hence so is
$F_{\epsilon}$.
Then by  the Vietoris-Begle Theorem \cite{Spanier},
we see that the map $\Psi_d^\epsilon$ induces an isomorphism
$$
(\Psi_d^\epsilon)_* :H_k(\tilde{P}\setminus V_{\epsilon},R)
\stackrel{\cong}{\longrightarrow}
H_k(\Psi_d (\tilde{P}\setminus V_{\epsilon}),R)
$$
for any $k$ and any commutative ring $R$.
Because
any singular chain of $A_d(m,n)$ is supported in
$\tilde{P}\setminus V_{\epsilon}$ for some $\epsilon>0$,
we get the desired assertion for $\Phi_d$ 
by letting $\epsilon \to +0$.

The corresponding result for the map $\Gamma_d:\tilde{A}_d(m,n)\to \Alg_d(\RP^m,\RP^n)$ can be proved either by an analogous argument or by the comparison of two fibration sequences:
$$
\begin{CD}
A_d(m,n) @>>> \tilde{A}_d(m,n) @>\tilde{ev}>> \RP^n
\\
@V{\Psi_d}VV @V{\Gamma_d}VV \Vert @.
\\
\Alg_d^*(\RP^m,\RP^n) @>>> \Alg_d(\RP^m,\RP^n) 
@>ev>> \RP^n
\end{CD}
$$
This completes the proof of Theorem \ref{thm: A4}.
\qed

\section{Appendix. The case $m=1$.}\label{section 6}

Theorem \ref{thm: IIIa} is unsatisfactory in one respect: we believe its statement ought to be analogous to that of Theorem \ref{thm: III} but have only been able to prove a weaker version. The stronger version would follow if we could replace the word \lq\lq homology\rq\rq\  by  \lq\lq homotopy\rq\rq\  in the statement of Theorem  \ref{thm: A4}. Unfortunately we have not been able to do this in general as we have not been able to prove that the spaces $\Alg_d^*(\RP^m,\RP^n)$ are simply connected for $m+2\le n$. However, it turns out that in the case $m=1$ we can use a more direct approach to show that the map $\Psi_d : A_{d}(1,n)\to \Alg_d^*(\RP^1,\RP^n)$ is a homotopy equivalence. 

We will exploit the convenient fact that spaces of tuples of polynomials in one variable can be identified with certain configuration spaces of points or particles in the complex plane. More exactly, we identify the space $A_d(1,n)$ with the space of $d$ particles of each of $n+1$ different colors, in the case $n=2$, say, red, blue and yellow, located in $\Bbb C = \Bbb R^2$ symmetrically with respect to the real axis and such that no three particles of different color lie at the same point on the real axis. 

\begin{center}
\includegraphics[width=3.7cm]{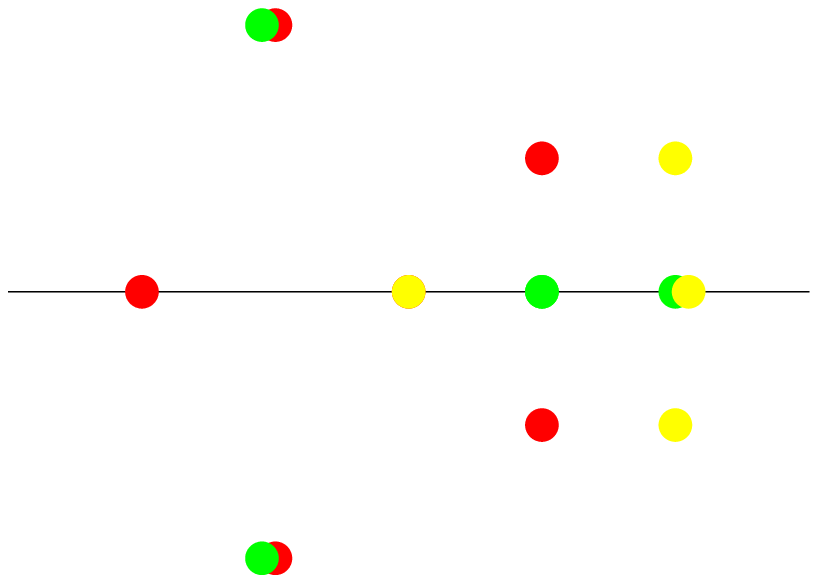} 
\end{center}
Note that off the real axis the particles are completely unrestricted. The space of algebraic maps $A_d(1,n)$ can also be thought of as a configuration space of  $k$ particles of each of $n+1$ different colors as above, where $k\leq d$, but with the additional property that when $n+1$ different particles meet (off the real line) they disappear. 

Finally, we need one more configuration space, introduced in \cite{Mo1}.  Let $T(d,n)$ denote the space of  no more than $k_i\le d$ particles of color $i$, where $k_i=d$ mod $2$, on the real axis , with the property that any even number of particles of the same color at the same point on the real axis vanish, and  of course, as before no $n+1$ particles of different colors lie at the same point. In other words,  $T(d,n)$ is a configuration space modulo 2. There is a natural map $\Phi :A_d(1,n)\to T(d,n)$ which factors through $\Psi_d:A_d(1,n) \to \Alg_d(\RP^1,\RP^n)$, 
$\Phi =Q_d\circ \Psi_d:A_d(1,n)
\stackrel{\Psi_d}{\longrightarrow} \Alg_d^*(\RP^1,\RP^n) \stackrel{Q_d}{\longrightarrow} T(d,n).$

 \begin{prop}[\cite {Mo1}, Proposition 2.1]
\label{prop: Mo}
The maps $\Psi_d$ and $Q_d$ above are homotopy equivalences. 
\end{prop}

A proof of this proposition is given in \cite{Mo1} but as it does not seem convincing  to us.
So we will give another one, which we believe to be correct.  More precisely, we shall prove the following:

 \begin{lemma}
\label{lemma: K}
The maps $Q_d$ and $Q_d \circ \Psi_d$  are quasi-fibrations with contractible fibers. 
\end{lemma}
From this it follows at once that $\Psi_d$ is a homotopy equivalence. 
Note that the proof of Proposition 2.1 in \cite{Mo1} is based on arguments claiming to prove that 
$Q_d\circ\Psi_d$ and $\Psi_d$ are homotopy equivalences which then is used to deduce that so is $Q_d$. However the argument claiming to prove that $\Psi_d$ is a homotopy equivalence makes no explicit use of any properties of the map 
$\Psi_d$, such as being a quasifibration etc. 

\begin{proof}[Proof of Lemma \ref{lemma: K}]
We first prove that the fiber of  $Q_d \circ\Psi_d$ (and $Q_d$) over any point in $T(d,n)$ is contractible. Consider a configuration in $T(d,n)$.
\par

\begin{center}
\includegraphics[width=4cm]{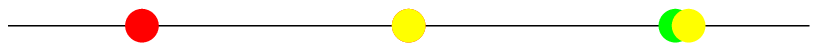} 
\end{center}
In the fiber over this configuration, all points in the upper half plane  are sent linearly to the fixed point $(1,0)$ and those in the lower half plane to $(-1,0)$. For a $2k$ or $2k+1$ fold particle lying on the real line, $k$ of the particles are moved to ${1,0}$ and $k$ are moved to  to $(-1,0)$ leaving $0$ or $1$ particles in place. 
This argument shows that the fibers of both  $Q_d\circ \Psi_d$ and $Q_d$ are contractible. 

\begin{center}
\includegraphics[width=4cm]{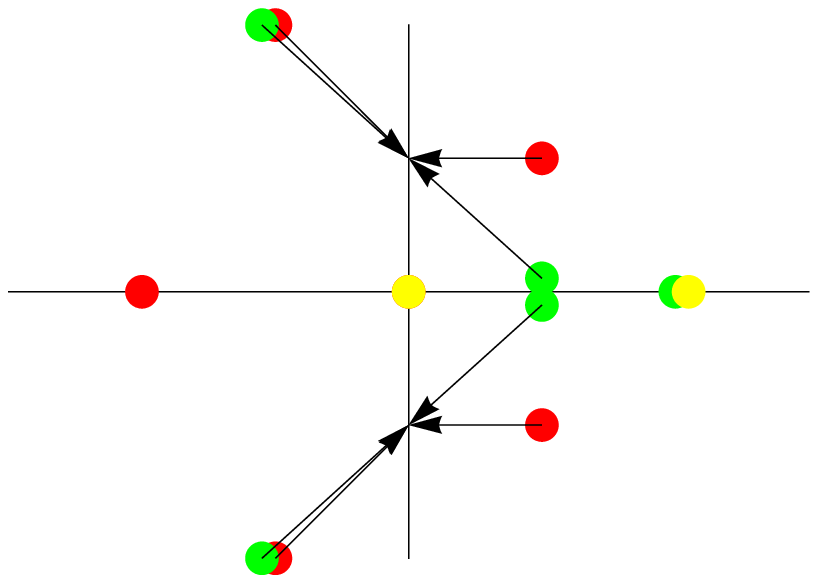} 
\end{center}
\par
Next, we will show that the maps  $Q_d \circ\Psi_d$ and $Q_d$ are both quasi-fibrations. Consider a point (configuration of particles)  in $T(d,n)$. By a \lq singular sub-configuration  for color $i$\rq\  we mean a collection $n$ particles of distinct colors other than the $i$-th color.  Intuitively, we think of a singular configuration for the $i$-th color as an obstacle through which a particle of the $i$-th color cannot pass. By a singular sub-configuration (without specifying the color) we mean a singular sub-configuration for any color. 
For non-negative integers $p_1,p_2,\dots, p_{n+1}$ with $\sum _i p_i\le d$, let  $T(d,n;p_1,\dots,p_{n+1})$ denote the subspace of $T(d,n)$ consisting of  configurations with precisely $p_i$ particles of  the $i$-th color. We start by proving that the restrictions of the maps $Q_d\circ \Psi_d$ and $Q_d$ to the pre-images of $T(d,n;p_1,\dots,p_{n+1})$ are quasi-fibrations.
 For any integer $k\ge 0$, let $T_k(d,n;p_1,\dots,p_{n+1})$ denote the subspace of $T(d,n;p_1,\dots,p_{n+1})$ consisting of configurations containing exactly $k$ singular sub-configurations. It is easy to see that the restriction of $Q_d \circ\Psi_d$ and $Q_d$ to the pre-image of $T_k(d,n;p_1,\dots,p_{n+1})$ is a locally trivial fiber bundle. Now we filter the space $T(d,n;p_1,\dots,p_{n+1})$ by closed subspaces $D_k(d,n;p_1,\dots,p_{n+1})$ of configurations containing $\ge k$ singular sub-configurations. 
 Set theoretic differences between these spaces are the spaces $T_k(d,n;p_1,\dots,p_{n+1})$ over which the maps are locally trivial fiber bundles and hence quasifibrations. Now we apply the Dold-Thom criterion (Lemma 4.3 of \cite{Hatch}). To do so we have to construct an open neighborhood of $D_{k+1}(d,n;p_1,\dots,p_{n+1})$ in $D_{k}(d,n;p_1,\dots,p_{n+1})$, a deformation of this neighborhood onto $D_{k+1}(d,n;p_1,\dots,p_{n+1})$, together with a corresponding covering neighborhood and a covering deformation required by the Dold-Thom criterion (since all the fibers are contractible the condition that the induced maps on the fibers are homotopy equivalences is automatically satisfied). Such neighborhoods and deformations are easy to describe intuitively. The set of points of the required  neighborhood of  $D_{k+1}(d,n;p_1,\dots,p_{n+1})$ consists of the points of  $D_{k+1}(d,n;p_1,\dots,p_{n+1})$ together with those configurations in $D_{k}(d,n;p_1,\dots,p_{n+1})$ with at least one non-singular sub-configuration of $n$ particles of different colors  contained in a \lq sufficiently small\rq\  interval. Here \lq sufficiently small\rq\  refers to the requirement that the particles in this sub-configuration be much nearer to each other than they are to any other particle and that the length of the minimal interval containing the sub-configuration be much less than the length of any interval containing a collection of $n+1$ particles of different color.  The deformation can now be defined by introducing a force of attraction between particles of different color (e.g. a force field satisfying  an inverse-square law). 
By induction on $k$ we show that the maps  $Q_d \circ\Psi_d$ and $Q_d$ are quasifibrations over $T(d,n;p_1,\dots,p_n)$. We now fix $p_1,p_2,\dots, p_n$ and filter the space $T(d,n)$ according to the number of points of the $n+1$-th color. The set theoretic differences between the terms of the filtration are precisely the spaces $T(d,n;p_1,\dots,p_{n+1})$ and we have already proved that the restriction of the maps $Q_d\circ \Psi_d$ and $Q_d$ to the inverse images of these spaces are quasi-fibrations. We apply again the Dold-Thom criterion. For this purpose we need to construct open neighborhoods of spaces of configurations  with no more than $k-2$ particles of the last color in the space of configurations of no more than $k$-particles of the last color. The method is again analogous to the one we used earlier. Our deformation will pull together pairs of particles of the last color which are very close by means of a gravitational force field between particles of the last color. For this purpose we must avoid hitting a \lq\lq singular point\rq\rq\    (a sub-configuration of $n-1$ particles of different colors). Again, it is easy to see that we can choose open neighborhoods and deformations with the right properties. This proves that the maps restricted to the pre-images of spaces with the number of particles of the first $n-1$ colors fixed are quasifibrations. Now we filter these spaces according to the number of particles of the last but one color. Proceeding by induction we see that $Q_d \circ\Psi_d$ and $Q_d$ are quasi-fibrations.
\end{proof}

\par\vspace{2mm}\par
\noindent{\bf Acknowledgements. }
The authors should like to take this opportunity to thank the referee
for his valuable suggestions.
In particular, the improvement of the assertions of 
Theorem \ref{thm: III}
and the proof of Theorem \ref{thm: A4} are essentially due to him.

The first author was supported by the Centre for Discrete Mathematics and its Applications, EPSRC award EP/D063191/1, during the last stages of this project, and
the third author is partially
supported by Grant-in-Aid for Scientific Research
(No. 19540068 (C)),
The Ministry of Education, Culture, Sports, Science and Technology, Japan.

\end{document}